\documentclass{amsart}
\usepackage{amsmath}
\usepackage{graphicx}
\usepackage{latexsym}
\usepackage{color}
\usepackage{amscd}
\usepackage[all]{xy}
\usepackage{enumerate}
\usepackage{hyperref}
\usepackage{subfigure}
\usepackage{comment}
\theoremstyle{plain}
\newtheorem{theorem}{Theorem}[section]
\newtheorem{corollary}[theorem]{Corollary}
\newtheorem{lemma}[theorem]{Lemma}
\newtheorem{proposition}[theorem]{Proposition}

\theoremstyle{definition}
\newtheorem{definition}[theorem]{Definition}

\newtheorem{conjecture}[theorem]{Conjecture}
\newtheorem{remark}[theorem]{Remark}
\newtheorem{question}[theorem]{Question}



\newcommand{\CPb}{\overline{\mathbb{CP}}{}^{2}}
\newcommand{\CP}{{\mathbb{CP}}{}^{2}}
\newcommand{\CL}{{\mathbb{CP}}{}^{1}}

\newcommand{\Z}{\mathbb{Z}}

\newcommand{\K}{{\rm K3}}

\newcommand{\Pa}{\partial}

\def\Diff{\operatorname{Diff}}

\def\id{\operatorname{id}}

\def\Crit{\operatorname{Crit}}

\def \x {\times}
\def \eu{{\text{e}}}

\begin{document}

\title[Unchaining surgery and symplectic $4$-manifolds]
{Unchaining surgery\,and \\ topology of symplectic $4$-manifolds}

\author[R. \.{I}. Baykur]{R. \.{I}nan\c{c} Baykur}
\address{Department of Mathematics and Statistics, University of Massachusetts, Amherst, MA 01003-9305, USA}
\email{baykur@math.umass.edu}

\author[K. Hayano]{Kenta Hayano}
\address{Department of Mathematics, Faculty of Science and Technology, Keio University, Yokohama, Kanagawa, 223-8522, Japan}
\email{k-hayano@math.keio.ac.jp}

\author[N. Monden]{Naoyuki Monden}
\address{Department of Mathematics, Faculty of Science, Okayama University, Okayama, Okayama 700-8530, Japan}
\email{n-monden@okayama-u.ac.jp}

\begin{abstract}
We study a symplectic surgery operation we call \emph{unchaining}, which effectively reduces the second Betti number and the symplectic Kodaira \linebreak dimension at the same time. Using unchaining, we give novel constructions of symplectic Calabi-Yau surfaces from complex surfaces of general type, as well as from rational and ruled surfaces via the natural inverse of this operation. Combining the unchaining surgery with others, which all correspond to certain monodromy substitutions for Lefschetz pencils, we provide further applications, such as a complete resolution of a conjecture of Stipsicz on the existence of exceptional sections in Lefschetz fibrations, new constructions of exotic symplectic \mbox{$4$-manifolds},  and inequivalent pencils of the same genera and the same number of base points on families of symplectic $4$--manifolds.  Meanwhile, we give a handy criterion for determining from the monodromy of a pencil whether its total space is spin or not.
\end{abstract}

\maketitle

\tableofcontents

\clearpage

\section{Introduction} 

Over the past couple of decades, many new surgeries, such as rational blowdown, generalized fiber sum, knot surgery, and Luttinger surgery, have been introduced to and successfully employed in the study of symplectic $4$--manifolds. These symplectic surgeries have typically emerged as more flexible versions of complex algebraic operations or more rigid versions of topological ones. A diverse family of symplectic surgeries, many of which seem to have very little history in the complex algebraic or topological worlds, come from excising a compact Stein subdomain and replacing it with a new one, which too induces the same contact structure on its convex boundary. In this article we undertake an extensive study of a symplectic surgery of this kind, which we call \emph{unchaining surgery}, and demonstrate several interesting features and applications of this surgery.

Since Lefschetz pencils and allowable Lefschetz fibrations are topological counterparts of closed symplectic $4$--manifolds and compact Stein domains \cite{Donaldson, GS, LP, AO}, these operations have natural interpretations in the framework of positive factorizations of Dehn twists in the mapping class group of fibers. Swapping Stein subdomains locally correspond to swapping positive factorizations of the monodromy of a supporting open book of the boundary contact $3$--manifold.  Moreover, whenever the Stein subdomain, as a positive allowable Lefschetz fibration, embeds into a Lefschetz pencil on the closed symplectic $4$--manifold, swapping the Stein subdomains has a global interpretation: it corresponds to a monodromy substitution in the global monodromy of the pencil. The unchaining surgery takes its name in this context from an important relation in the mapping class group of a surface: the \emph{chain relation}, which exchanges a product of Dehn twists along an odd number of chain of curves on a compact (sub)surface, with a pair of Dehn twists along the two boundary components of the tubular neighborhood of this chain; see Lemma\ref{T:chain relation}. (And the analogous surgery corresponding to the chain relation for an even number of chains shares very similar features; see Remark~\ref{evenchains}.) 


In all the examples we produce in this article, we will perform unchaining from the vantage point of substitutions in positive factorizations, and aim to preserve the global Lefschetz fibration structures, so we can determine the \emph{Kodaira dimension} of the resulting symplectic $4$--manifolds ---which is a notion that measures the positivity of the symplectic canonical class, extended from the case of compact complex surfaces, and is a diffeomorphism invariant \cite{Li1}. We will read off the symplectic Kodaira dimension using the additional data we derive on their exceptional multisections (symplectic $(-1)$--spheres which intersect all the fibers positively); see Theorem~\ref{KodFromLF} and \cite{Sato, BaykurHayano}. Precise definitions and background results for all of the above are given in Sections~\ref{Sec:preliminaries} and\ref{Sec:unchaining}.

Our first application concerns the topology of \emph{symplectic Calabi-Yau} surfaces, which constitute symplectic $4$--manifolds of Kodaira dimension zero, up to finite covers. Recall that a  symplectic Calabi-Yau surface is a symplectic \mbox{$4$--manifold} with a trivial canonical class, similar to a complex Calabi-Yau surface. Up to date, the only known examples of symplectic Calabi-Yau surfaces are torus bundles over tori and complex $\K$ surfaces, and any symplectic Calabi-Yau surface is known to necessarily have the rational homology type of these manifolds \cite{Li4, Bauer}. It is an alluring open question whether this fairly short list of  symplectic Calabi-Yau surfaces is complete up to diffeomorphisms, and in particular if there is no symplectic Calabi-Yau surfaces with $b_1=0$, that is not diffeomorphic to a $\K$ surface \cite{Li1, Donaldson_SCY, FriedlVidussi}. As Tian-Jun Li points out in \cite{LiNew}, the poor state of this problem seems to also stem from the lack of any new constructions of symplectic Calabi-Yau surfaces; many widely used symplectic surgeries mentioned in the beginning of our article are seen not to yield any new SCYs, except for trivial cases \cite{HoLi, LiNew, UsherKodaira, Dorfmeister2}.

In Section~\ref{Sec:newpencils} we construct, for each $g \geq 3$, a variety of new positive factorizations yielding  genus--$g$ Lefschetz pencils on symplectic $4$--manifolds with any Kodaira dimension, via careful applications of unchaining, starting with holomorphic pencils on complex surfaces of general type; see Theorems~\ref{thm:3} and \ref{T:Kodaira dim X_g(i)}. These constructions demonstrate how unchaining surgery can produce symplectic Calabi-Yau surfaces (Corollary~\ref{SCYK3s}) from complex surfaces of general type, or from rational or ruled surfaces via the natural inverse of the operation.  It is plausible that this symplectic surgery can  be performed as a complex operation in favorable situations, as an inverse to contracting a $(-2)$--curve and then taking a minimal resolution; see Remark~\ref{contracting}. As Simon Donaldson proposes in \cite{DonaldsonMCG}, an analysis of monodromies of pencils on SCYs may shed light on their geometry and topology, and thus, we record our examples with this extra feature:

\begin{theorem} \label{MainThm1}
For each $g \geq 3$,  there is a symplectic genus-$g$ Lefschetz pencil $(K_g, k_g)$, where $K_g$ is a symplectic Calabi-Yau homotopy $\K$ surface, and $k_g$ has the monodromy factorization: 
\begin{align*}
& t_{\delta_{g-1}} \cdots t_{\delta_2} t_{\delta_1} t_{\delta_{g-1}^\prime} \cdots t_{\delta_2^\prime} t_{\delta_1^\prime}  \\
=& \begin{cases}
t_{x_{g}} \cdots t_{x_1} t_{x_{g}^\prime} \cdots  t_{x_1^\prime} (t_{c_1} t_{c_2} t_{c_3})^{8} t_{d_4}  \cdots t_{d_{2g+1}} t_{e_{2g+1}} \cdots  t_{e_4} & (g:\mbox{odd}) \\
t_{x_{g}} \cdots t_{x_1} t_{x_{g}^\prime} \cdots  t_{x_1^\prime} (t_{c_1} t_{c_2} t_{c_3})^{6} (t_{c_3} t_{c_2} t_{c_1})^2 t_{d_4} \cdots t_{d_{2g+1}} t_{e_{2g+1}} \cdots  t_{e_4} & (g:\mbox{even})
\end{cases}
\end{align*}
in  $\Gamma_g^{2g-2}$, where the curves $\delta_j, \delta'_j, c_j, d_j, e_j, x_j,x'_j$ are as in  Figures~\ref{F:curves1}\ref{vanishingcycles}\ref{sectioncurves}. 
\end{theorem}

\noindent These complement the explicit monodromies of pencils on symplectic Calabi-Yau surfaces with $b_1>0$ given in the works of Ivan Smith, the first author, and Noriyuki Hamada and the second author in  \cite{SmithTorus, BaykurGenus3, HamadaHayano}. The symplectic $4$--manifolds $K_g$ have the same fundamental group and Seiberg-Witten invariants as complex $\K$ surfaces,  and we do not know at this point whether they are all diffeomorphic to them.

Our second application concerns a riveting conjecture of Andras Stipsicz on the existence of exceptional sections in \emph{fiber-sum indecomposable} Lefschetz fibrations. In \cite{Stipsicz_indecomposability}, having proved the converse, Stipsicz conjectured that any  Lefschetz fibration, which cannot be expressed as a fiber-sum of any two non-trivial fibrations, always admits an exceptional section ---an affirmative answer to which would mean that any Lefschetz fibration is a fiber-sum of blown-up pencils. This conjecture was shown to fail in genus--$2$ by just a handful of examples : a holomorphic fibration on a blown-up $\K$  surface  by Auroux (which was the first counter-example, as observed by Yoshihisa Sato in \cite{Sato2010}), a symplectic fibration on a homotopy Enriques surface by the first two authors of this article \cite{BaykurHayano}, a holomorphic fibration with $6$ irreducible and $7$ reducible fibers by Xiao \cite{Xiao} (as observed to be a counter-example by the first author), and one more (if different than Xiao's), by an implicit argument  in \cite{AkhmedovMonden}. Only two more counter-examples, a pair of genus--$3$ fibrations on symplectic Calabi-Yau homotopy $\K$ surfaces were provided again in \cite{BaykurHayano}, and there has been no known examples for any $g \geq 4$, up to date.

In Section~\ref{Sec:spin}, we prove that there is in fact no stable range for $g$, where this conjecture may hold:

\begin{theorem}\label{MainThm2}
For any $g\geq 2$, there exists a genus--$g$ fiber-sum indecomposable  Lefschetz fibration without any exceptional sections. 
\end{theorem}

\noindent The  most challenging part of  constructing such examples is to locate all the exceptional spheres in the symplectic $4$--manifold so as to know for sure that there is none that can be a section. We will produce all our genus $g \geq 4$ examples by applying the braiding lantern substitution of \cite{BaykurHayano} to the explicit positive factorizations of a family of pencils with Kodaira dimension one we obtain via unchaining (Theorem~\ref{T:counterex Stipsicz}). Along with the already existing examples of $g=2,3$, this then completely resolves Stipsicz's conjecture (noting that by the well-known classification of genus--$1$ fibrations, the conjecture does hold in the remaining case). 

In the course of the proof of the above theorem, we establish another result, which would be of independent interest: a complete characterization of when the total space of a Lefschetz pencil is spin, in terms of its monodromy factorization (Theorem~\ref{T:condition spin LP}). This extends Stipsicz's earlier work in \cite{Stipsicz_spin} for Lefschetz fibrations (no base points), and this generalization is applicable to any symplectic $4$--manifold, since a given symplectic $4$--manifold  may not admit a Lefschetz fibration, but it always admits a pencil (with base points) by Donaldson.

\begin{theorem} \label{MainThm3}
Let $(X,f)$ be a genus--$g$ Lefschetz pencil with a monodromy factorization \ $t_{c_1}\cdots t_{c_n}=t_{\delta_1}\cdots t_{\delta_p}$. Then $X$ admits a spin structure if and only if there exists a quadratic form $q:H_1(\Sigma_g^p;\mathbb{Z}/2\mathbb{Z})\to \mathbb{Z}/2\mathbb{Z}$ with respect to the intersection pairing of $H_1(\Sigma_g^p;\mathbb{Z}/2\mathbb{Z})$ such that $q(c_i)=1$ for any $i$ and $q(\delta_j)= 1$ for some $j$. 
\end{theorem}

In the final section, Section~\ref{Sec:further}, we  will present two more applications, one regarding the topology of symplectic $4$--manifolds (Theorem~\ref{exotic}, and one regarding that of pencils (Theorem~\ref{thm:5}). All the examples we construct therein come from combining unchaining and rational blowdown surgeries that respect the fibration structures. 

Constructions of  small simply-connected symplectic $4$--manifolds with $b_2^+ \leq 3$ via rational blowdowns has a fairly long and rich history, pioneered by the works of Fintushel--Stern and Jongil Park \cite{FSrationalblowdown, Park0, FSPark, Park}. Similar constructions via  monodromy substitutions in positive factorizations of Lefschetz pencils was first given by Endo and Gurtas \cite{EndoGurtas}, who observed that lantern substitutions (see Lemma~\ref{lantern}) amount to a rational blowdown of a symplectic $(-4)$--sphere ---which since then, has been extended to many other substitutions corresponding to blowdowns of more general configurations of spheres \cite{EndoEtal, GayMark, KarakurtStarkston}. 
The hardship of the latter approach is to have explicit positive factorizations of pencils that contain an enough number of lantern configurations for rational blowdowns, and was so far successfully applied to genus--$2$ Lefschetz pencils in \cite{EndoGurtas} and \cite{AkhmedovPark}. We demonstrate that, through unchaining, we do get such useful monodromy factorizations, from which,  we can for example obtain new symplectic genus--$3$ Lefschetz fibrations with exotic total spaces: 

\begin{theorem}\label{MainThm4}
There are  genus--$3$ Lefschetz fibrations $(X_j, f_j)$, for   $j=0,1,2, 3$, where each $X_j$ is a minimal symplectic $4$--manifold homeomorphic but not diffeomorphic to $3\CP \sharp (19-j)\CPb$, and  each $f_{j+1}$ has a monodromy factorization obtained from that of $f_{j}$ by a lantern substitution.
\end{theorem}

\noindent Moreover, by the techniques of \cite{BaykurHayano}, one can explicitly describe the symplectic canonical class of each $X_j$ as a multisection of $f_j$; see Remark~\ref{whereiscan}. 

Our final application is on the diversity of Lefschetz pencils and fibrations of the same genera on a given symplectic $4$--manifold (up to equivalence of pencils/fibrations through self-diffeomorphisms of the $4$--manifold and the base, commuting the maps). Examples of inequivalent Lefschetz \emph{fibrations} on  a blow-up of $T^2 \x \Sigma_2$, whose fibers have different divisibility in homology were discussed by Ivan Smith in this thesis, and several inequivalent fibrations on homotopy elliptic surfaces, distinguished by their monodromy groups,  were produced by Jongil Park and Ki-Heon Yun in  \cite{ParkYun, ParkYun2}. In \cite{BaykurInequivalentLF} and \cite{BaykurInequivalentRR}, building on Donaldson's existence result and the doubling construction for pencils, the first author established that in fact \emph{any} symplectic $4$--manifold, possibly after blow-ups, admits inequivalent Lefschetz pencils and fibrations of arbitrarily high genera. Also see \cite{BaykurHayano} and \cite{Hamada} for inequivalent \emph{pencils} on homotopy $\K$ surfaces and on ruled surfaces, respectively. 
Here we produce many more examples of inequivalent Lefschetz pencils and fibrations, notably with explicit positive factorizations: 

\begin{theorem}\label{MainThm5}
For any $g\geq 3$ and $i=0,1,2,\ldots,g-1$, there are pairs of inequivalent relatively minimal  genus--$g$ Lefschetz pencils $(Y_g(i), h^j_g(i))$, $j=1,2$, and inequivalent Lefschetz fibrations on their blow-ups. For any $g\geq 3$, there are pairs of  inequivalent  relatively minimal genus--$g$ Lefschetz pencils on once blown-up elliptic surface $E(1) \, \sharp \, \CPb \cong \CP\, \sharp \, 10\, \CPb$, with different number of reducible fibers.
\end{theorem}

\noindent The examples of inequivalent pencils we obtain here are on a family of symplectic $4$--manifolds, whose symplectic Kodaira dimensions run through $-\infty, 0$ and $1$.

\vspace{0.2in}
\noindent \textit{Acknowledgements. } The results of this article were announced in several seminar talks and conferences since 2015, including the Great Lakes Geometry Conference (Ann Arbor, 2015), the SNU Topology Winter School (Seoul, 2015), and the Four Dimensional Topology Workshop  (Osaka, 2017), and we thank the organizers and participants  for their interest and comments. R. I. B. was partially supported by the NSF grant DMS-1510395, K.H. by JSPS KAKENHI (Grant Number 17K14194), and N. M. by  JSPS KAKENHI (Grant Number 16K17601).  
\vspace{0.4in}

\medskip
\section{Preliminaries} \label{Sec:preliminaries} 

Here we review the main notions and background results used throughout the paper. Manifolds in this paper are assumed to be smooth, connected and oriented, unless otherwise stated.

\subsection{Lefschetz fibrations and pencils} \

Let $X$ and $\Sigma$ be compact manifolds of dimension $4$ and $2$, respectively. A smooth map $f\colon X\to \Sigma$ is called a \textit{Lefschetz fibration} if the critical locus $C=\text{Crit}(f)$ is a discrete set such that

\begin{itemize}
\item
for any $x\in \Crit(f)$ there are complex charts $(U,\varphi)$ at $x$ and $(V,\psi)$ at $f(x)$, compatible with the orientation of $X$ and $\Sigma$, so that
\[
\psi\circ f\circ \varphi^{-1}(z,w)=z w,
\]
\item
the restriction $f|_{f^{-1}(f(C)}$ is a surface bundle, and
\item
the restriction $f|_{\Crit(f)}$ is injective. 

\end{itemize}

For $B\subset X$ a \emph{non-empty} finite set of points, a smooth map $f:X\setminus B\to \CL$  is called a \textit{Lefschetz pencil} if $f|_{X\setminus \nu B}$ is a Lefschetz fibration, where $\nu B$ is the union of balls centered at points in $B$, and for any $b\in B$ there exist a complex chart $(U,\varphi)$ compatible with the orientation of $X$, and an orientation-preserving self-diffeomorphism $\Phi:\CL\to \CL$ so that:
\[
\Phi\circ f\circ \varphi^{-1}(z,w) = [z:w]. 
\]
Each point in $B$ is called a \textit{base point of $f$}. 

We will denote a Lefschetz fibration or a pencil simply as the pair $(X, f)$. For a Lefschetz fibration or pencil $(X,f)$, the set of critical values $f(C) \subset \Sigma$ is discrete. So the genus of the closure of a regular fiber $\overline{f^{-1}(q)}$ does not depend on the regular value $q$, and is called the \textit{genus of $(X,f)$}. In this paper, we will always assume that $(X,f)$ is \textit{relatively minimal}, that is, for any point $q \in f(X)$, the closure $\overline{f^{-1}(q)}$ does not contain any $(-1)$-spheres. 

Lastly, an \emph{allowable Lefschetz fibration} is a Lefschetz fibration with base $\Sigma=D^2$, where the fibers have non-empty boundaries, and for any point $q \in f(X)$, $f^{-1}(q)$ does not contain any closed surfaces.

\subsection{Positive factorizations} \

Let $\Sigma_g^n$ be a compact genus-$g$ surface with $n$ boundary components. 
We denote by $\Gamma_g^n$ the mapping class group of $\Sigma_g^n$: 
\[
\Gamma_g^n = \pi_0\left(\left\{\psi\in \Diff^+(\Sigma_g^n)\left|\psi|_{\partial \Sigma_g^n}=\id\right.\right\}\right).  
\]
Let $(X,f)$ be a Lefschetz fibration or pencil over $\Sigma=\CL$ with $l$ critical points and $n$ base points (where $n=0$ if it is a fibration). Take a regular value $p_0\in \CL$ and simple paths $\gamma_1,\ldots,\gamma_l$ from $p_0$ to critical values which are mutually disjoint except at $p_0$. 
Suppose that $\gamma_1,\ldots,\gamma_l$ appear in this order when we go around $p_0$ counterclockwise. 
We denote by $\alpha_i$ a loop with base point $p_0$ which first goes along $\gamma_i$, then goes around a critical value and goes back to $p_0$ along $\gamma_i$. 
Kas \cite{Kas} showed that a monodromy of $f$ along $\alpha_i$, which can be regarded as an element of $\Gamma_g^n$ under an identification $\Sigma_g^n \cong f^{-1}(p_0)\setminus \nu B$, is a \emph{positive Dehn twist} $t_{c_i}$ along a simple closed curve $c_i\subset \Sigma_g^n$, called the \textit{vanishing cycle} of $f$ associated to $c_i$ by $\gamma_i$.

We can deduce from the local description of $f$ around the base points that the monodromy of $f$ along a concatenated loop $\alpha_1\cdots \alpha_l$ is equal to $t_{\delta_1}\cdots t_{\delta_n}\in\Gamma_g^n$, where $\delta_1,\ldots,\delta_n\subset \Sigma_g^n$ are simple closed curves parallel to respective boundary components. 
We therefore obtain the following \emph{positive factorization} of the boundary multi-twist $t_{\delta_1}\cdots t_{\delta_n}$ in $\Gamma_g^n$: 
\[
t_{c_l}\cdots t_{c_1} = t_{\delta_1}\cdots t_{\delta_n},
\]
which is called the \textit{monodromy factorization} of the Lefschetz pencil $f$. When there are no base points, i.e. when $n=0$, this is a factorization of identity in $\Gamma_g$.  Note that relative minimality of a Lefschetz fibration or pencil implies that no $c_i$ is null-homotopic. 
In the case of an allowable Lefschetz fibration, where the base is $D^2$ instead of $\CL$, one derives a positive factorization of an element in $\Gamma_g^n$, $n>0$, which doesn't need to be a boundary multi-twist, and no $c_i$ is null-homologous. 

Conversely, given such a positive factorization as above, one can build  a genus--$g$ Lefschetz  pencil $(X,f)$ with $l$ critical points and $n$ base points, where $X$ is a symplectic $4$--manifold  \cite{GS}.  Similarly, a positive factorization of any mapping class $\mu$ in $\Gamma_g^n$ with $n>0$, where no Dehn twist curve $c_i$ is null-homologous, one can build a genus--$g$ allowable Lefschetz fibration $(X,f)$, where $X$ is a Stein domain \cite{LP, AO}. On the boundary of $X$, this fibration induces an open book with monodromy $\mu$, which supports the natural contact structure induced by the Stein structure on $X$ (either as the maximal distribution of the complex structure restricted to the boundary, or as the kernel of contact $1$--form which is the primitive of the symplectic structure around the boundary). Moreover, from a given positive factorization for a Lefschetz fibration or pencil, we can obtain another one by substituting a subword of the factorization with another product of positive Dehn twists. In the following sections we will construct several Lefschetz fibrations and pencils by this procedure, sometimes called as \emph{monodromy substitution}. In these constructions, we will repeatedly use the following well-known relations (see e.g. \cite{FM_MCG}):

\begin{lemma}[Lantern relation] \label{lantern}

	Let $\delta_1$, $\delta_2$, $\delta_3$ and $\delta_4$ be the four boundary curves of $\Sigma_0^4$ and let $x$, $y$ and $z$ be the interior curves as shown in Figure~\ref{LR}. 
	Then, we have the \emph{lantern relation} in $\Gamma_0^4$:
	\begin{align*}
	t_{\delta_4} t_{\delta_3} t_{\delta_2} t_{\delta_1} = t_z t_y t_x. 
	\end{align*}

\begin{figure}[hbt]
	\centering
	\includegraphics[scale=.45]{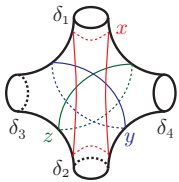}
	\caption{The curves $x$, $y$, $z$ on $\Sigma_0^4$.}
	\label{LR}
\end{figure}

\end{lemma}

\begin{lemma}[Odd chain relation]\label{T:chain relation}
	Let $d_1, d_2,\ldots,d_{2h+1}$ be simple closed curves on $\Sigma_g^n$ such that $d_i$ and $d_j$ are disjoint if $|i-j|\geq 2$ and that $d_i$ and $d_{i+1}$ intersect at one point. Then, a regular neighborhood of $d_1\cup d_2 \cup\cdots \cup d_{2h+1}$ is a subsurface of genus $h$ with two boundary components, $b_1$ and $b_2$. We then have the \emph{odd chain relation}  in $\Gamma_g^n$:
	\begin{align*}
	(t_{d_1}t_{d_2}\cdots t_{d_{2h+1}})^{2h+2} = t_{b_1} t_{b_2}. 
	\end{align*}
\end{lemma}

\medskip
\subsection{Symplectic Kodaira dimension and Calabi-Yau surfaces}\label{KodSection} \

A symplectic $4$-manifold is said to be \emph{minimal} if it contains no symplectic sphere with self-intersection $-1$. Let $(X_{\min},\omega_{\min})$ be a minimal  symplectic $4$-manifold
obtained from a given closed symplectic manifold $(X,\omega)$ by blowing-down all the symplectic spheres. For $K_{\min}$ the canonical class of $(X_{\min},\omega_{\min})$, 
the \emph{symplectic Kodaira dimension} $\kappa(X)$ of $(X,\omega)$ is defined as follows: 
	\[
	\kappa(X) = \begin{cases}
	-\infty &  \textit{if} \ \   K_{\min}\cdot [\omega_{\min}] < 0 \mbox{ or }K_{\min}^2 <0\\
	\ \ 0 & \ \ \ \ K_{\min}\cdot [\omega_{\min}] =K_{\min}^2 =0\\
	\ \ 1 & \ \ \ \ K_{\min}\cdot [\omega_{\min}] > 0\mbox{ and }K_{\min}^2 =0\\
	\ \ 2 & \ \ \ \ K_{\min}\cdot [\omega_{\min}] > 0 \mbox{ and }K_{\min}^2 >0. 
	\end{cases}
	\] 
It turns out that $\kappa(X)$ is not only independent of the associated minimal symplectic manifold $(X_{\min},\omega_{\min})$, but also the symplectic structure $\omega$ on $X$ (\cite{Li1}) ---so it is a diffeomorphism invariant.

We will make use of the following criterion for a symplectic $4$-manifold to have a specific Kodaira dimension:

\begin{theorem}[Kodaira dimension from monodromy factorizations, {\cite{Sato, BaykurHayano, BaykurGenus3}}] \label{KodFromLF}
Let $X$ be a symplectic $4$-manifold which is the total space of a genus $g \geq 2$ Lefschetz fibration with a monodromy factorization which lifts to $\Gamma_g^n$ as a factorization of the boundary multi-twist, i.e.
\[ t_{c_l} \cdot \ldots \cdot t_1 = t_{\delta_n} \cdot \ldots \cdot t_{\delta_1} . \]
\begin{enumerate}
\item
The kodaira dimension $\kappa(X)$ is equal to $-\infty$ if $n > 2g-2$, \

\item
$\kappa(X) = 0$ if $n= 2g-2$ and $b^+(X) \neq 1$, 

\item
$\kappa(X)=1$ if $n = 2g-3$ and $b^+(X) >3$, 
\end{enumerate}
\end{theorem}
\noindent

Note that the  topological invariant $b^+(X)$ in the theorem can be read off from the positive factorization. Since $\pi_1(X)$ is isomorphic to the quotient of $\pi_1(\Sigma_g)$ by $N(c_1, \ldots, c_l)$, the subgroup normally generated by $c_i$s, we can calculate $b_1(X)$ as the rank of the abelianization of $\pi_1(X)$. Then calculating  $\eu(X)= 4-4g+ l$ and $\sigma(X)$ algorithmically (e.g. by \cite{EN}), we can obtain  $b^+(X)$ through the equality $b^+(X)=(\eu(X)-2+2b_1+\sigma(X))/2$.

\medskip
\section{Unchaining operation} \label{Sec:unchaining}

In this section we will discuss the \emph{unchaining} operation, which is a symplectic surgery that can be interpreted (at least locally) as a monodromy substitution in a Lefschetz fibration. This surgery will play a key role in all our constructions throughout the paper. 

Let $c_1,\ldots,c_{2g+1},\delta_1,\delta_1'\subset \Sigma_g^2$ be simple closed curves shown in Figure~\ref{F:curves1}.  By Lemma~\ref{T:chain relation}, the following relation holds in $\Gamma_g^2$:
\[
(t_{c_1}\cdots t_{c_{2g+1}})^{2g+2} = t_{\delta_1}t_{\delta_1'}. 
\]

\begin{figure}[hbt]
 \centering
     \includegraphics[width=12cm]{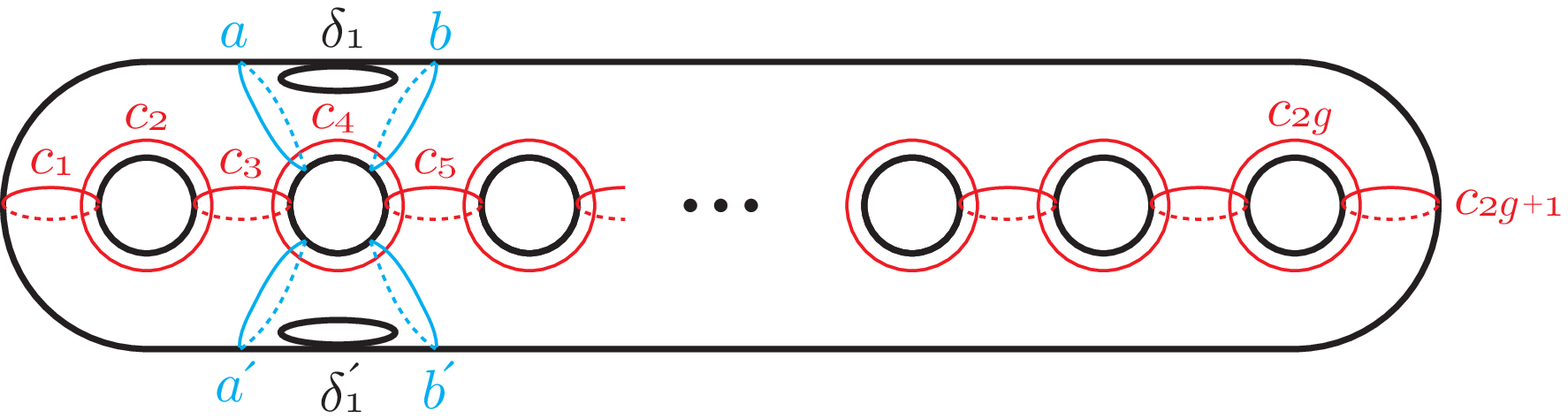}
     \caption{The curves $a,a^\prime,b,b^\prime$ and $c_i$ $(i=1,\ldots,2g+1)$ on $\Sigma_g^2$.}
     \label{F:curves1}
\end{figure}

\noindent Denote the total space of the allowable Lefschetz fibration corresponding to the left and right hand sides of the relation above by $V_g$ and $V_g'$, respectively. One can describe the compact manifolds $V_g$ and $V_g'$ by handlebody diagrams  in Figure~\ref{F:diagramV_gV_g'}.

\begin{figure}[htbp]
\centering
\subfigure[$V_g$ ]{
\includegraphics[width=125mm]{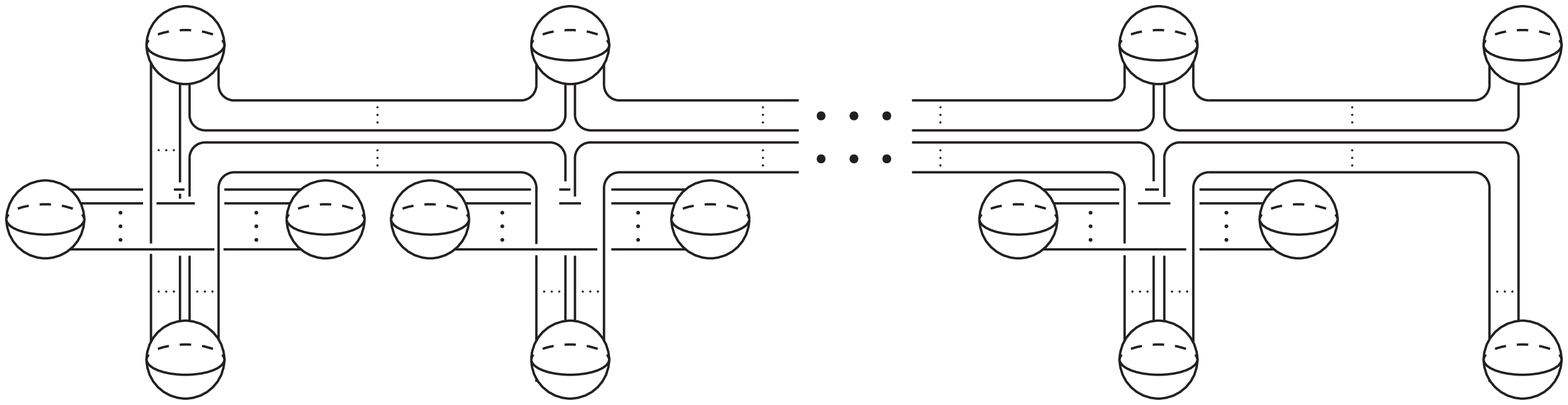}
}
\subfigure[$V_g'$. ]{
\includegraphics[width=125mm]{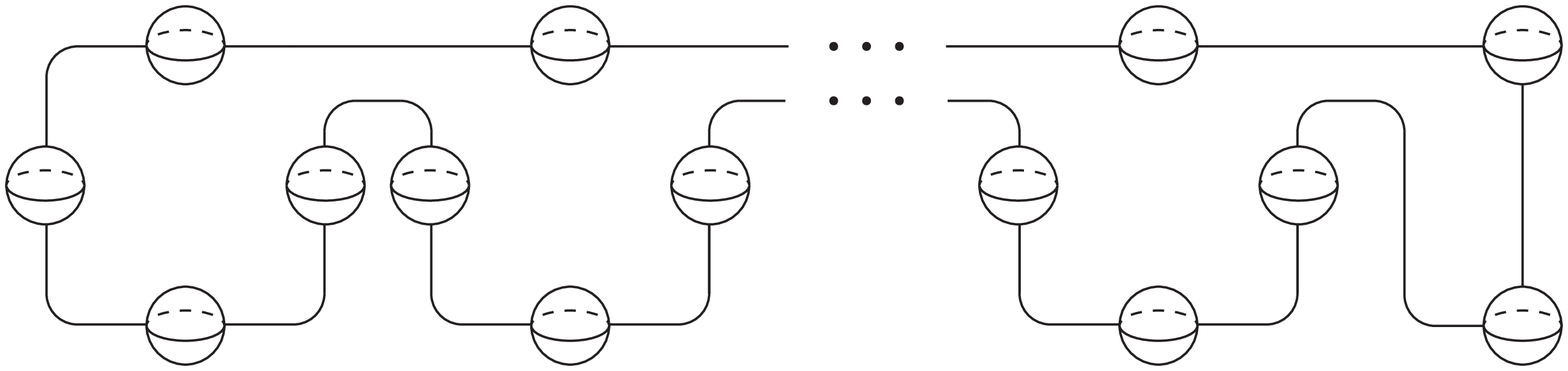}
}
\caption{Handlebody diagrams of $V_g$ and $V_g'$. 
All the $2$--handles in both figures have $(-1)$-framing. }
\label{F:diagramV_gV_g'}
\end{figure}

Taking the Stein structures associated to their allowable Lefschetz fibrations above, we regard $V_g$ and $V'_g$ as Stein domains. (It is a straightforward exercise to turn the handle diagrams  in Figure~\ref{F:diagramV_gV_g'} into diagrams of Stein handle decompositions following \cite{AO}.) Since these fibrations induce the same open book monodromy  on their boundaries, the contact structures induced on $\partial V_g$ and $\partial V_g'$ by the underlying symplectic structures on $V_g$ and $V'_g$ are contactomorphic. Thus, if $V_g$ is a Stein subdomain of a symplectic $4$--manifold $(X, \omega)$, we can excise $V_g$ and glue in $V'_g$ symplectically via any contactomorphism $\varphi: \Pa V_g \to \Pa V'_g$, perhaps after scaling the symplectic form on it (see e.g. \cite{Etnyre}). We will call this particular symplectic cut-and-paste operation \emph{unchaining}:

\begin{definition}\label{unchaining}
Let $(X, \omega)$ be a symplectic $4$--manifold, containing $V_g$ as a Stein submanifold. The symplectic $4$--manifold $(X', \omega')$, where $X'=(X\setminus V_g) \cup_\varphi V_g'$, with $\omega'|_{X \setminus V_g} = \omega_{X \setminus V_g}$ and containing $V_g'$ as a Stein subdomain is then said to be obtained by \emph{unchaining} $(X, \omega)$ along $V_g\subset X$ (or by \emph{$g$--unchaining}, whenever we would like to be specific about the size) . 
\end{definition}

This local surgery realizes a monodromy substitution when $(X, \omega)$ is the total space of a symplectic Lefschetz fibration or pencil $f$, whose monodromy factorization contains $(t_{d_1}\cdots t_{d_{2h+1}})^{2h+2}$ as a subword, where $d_1,\ldots, d_{2h+1}$ are simple closed curves in a reference fiber of $f$ satisfying the condition in Lemma~\ref{T:chain relation}. 
We can then apply a \emph{$C_{2h+1}$-substitution} to the monodromy factorization of $f$, that is, we can substitute a subword $(t_{d_1}\cdots t_{d_{2h+1}})^{2h+2}$ in the factorization with $t_{b_1}t_{b_2}$, where $b_1,b_2$ are curves given in Lemma~\ref{T:chain relation}. 
Since the substitution does not change the right-hand side of the factorization, we obtain another symplectic Lefschetz fibration or pencil $f'$, whose total space is a symplectic $4$--manifold $(X, \omega')$ obtained by unchaining $(X, \omega)$ along $V_g$. (In this case we choose $\varphi:\Pa V_g \to \Pa V_g'$ to be a fiber-preserving contactomorphism between the boundary open books, which is identity along the fibers.) We will build all our examples in this paper from this perspective. 

\medskip
In the remainder of this section, we will explain how the Euler characteristic, the signature and the fundamental group of $X$ and $X'$ are related by the unchaining operation along $V_g \subset X$. 

Let's start with the invariants for the subdomains $V_g$ and $V'_g$. From the handle decompositions given above, it is easy to see that the Euler characteristics $e(V_g)$ and $e(V_g')$ are respectively equal to $2(2g^2+2g+1)$ and $-2g +2$. The manifold $V_g$ is the complement of the union $U$ of neighborhoods of a regular fiber and two $(-1)$-sections in a \emph{hyperelliptic} Lefschetz fibration over $S^2$ with $(2g+1)(2g+2)$ irreducible singular fibers and no reducible ones. Therefore, we can deduce from \cite{EN} that the signature of the hyperelliptic fibration over $S^2$ is $-2(g+1)^2$. Since the signature of $U$ is $-1$, by the Novikov additivity,  the signature of $V_g$ is equal to $-2(g+1)^2+1$. On the other hand, since $V_g'$ is a disk bundle over a genus-$g$ surface with Euler number $-2$, the signature $\sigma(V_g')= -1$.  Now by the additivity of Euler characteristic and signature for gluing compact $4$--manifolds along $3$--manifolds, we can conclude that:

\begin{proposition}\label{T:topinv_unchaining1}
Let $X'$ be obtained by unchaining $X$ along $V_g \subset X$. Then,
\begin{align*}
e(X') & = e(X) -2g (2g+3), \\
\sigma(X') &= \sigma(X) + 2g(g+2).  
\end{align*}
\end{proposition}

\smallskip
The relation between the fundamental groups of $X$ and $X'$ is more involved, since in general it non-trivially depends on the fundamental group of the complement of $V_g$ in $X$.
The manifold $\Pa V_g$ (and thus $X$) contains a surface $\Sigma\cong \Sigma_g^2$ which is a fiber of the boundary open book. We take a generating set $\{a_1,b_1,\ldots,a_g,b_g,d\}$ of $\pi_1(\Sigma)$ as shown in Figure~\ref{F:generator_pi1}. 
\begin{figure}[htbp]
\centering
\includegraphics[width=70mm]{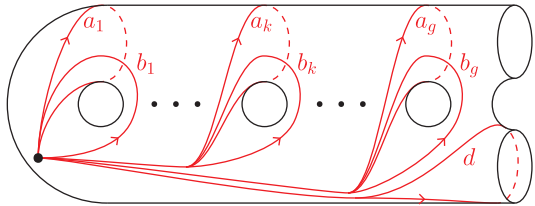}
\caption{A generating set of a fiber of the open book of $\Pa V_g$. }
\label{F:generator_pi1}
\end{figure}
We also denote the elements in $\pi_1(\overline{X\setminus V_g})$ represented by loops in Figure~\ref{F:generator_pi1} by $a_i, b_j$, and so on. Recall that $V'_g$ admits a handle decomposition with one $0$--handle, $2g+1$ $1$--handles, and two $2$--handles as in Figure~\ref{F:diagramV_gV_g'}. Now, $X'$ is obtained from $\overline{X\setminus V_g}$ by attaching this handlebody upside down to it. Then, the cellular decomposition induced by this handle decomposition implies that $X'$ is obtained from $\overline{X\setminus V_g}$ by attaching two $2$--cells to an open book fiber $\Sigma$ of $\partial(\overline{X\setminus V_g}) = \partial V'_g$ along loops freely homotopic to $d$\,and {$d^{-1}\cdot\prod_{i=1}^g [a_i,b_i]$}, and then attaching $2g+1$ $3$--cells and a $4$--cell.  By a standard consequence of Seifert Van-Kampen, the $2$--cells yield further relations via their attaching maps, while the higher dimensional cells have no effect on $\pi_1$. We therefore get:

\begin{proposition}\label{T:topinv_unchaining2}

The fundamental group of $X'$ is isomorphic to
\begin{center}
$\pi_1(\overline{X\setminus V_g}) \, / \, N(d,\prod_{i=1}^g [a_i,b_i])$
\end{center}
where $N(d,\prod_{i=1}^g [a_i,b_i])$ is the subgroup of $\pi_1(\overline{X\setminus V_g})$ normally generated by $d$,\,and $\prod_{i=1}^g [a_i,b_i]$. In particular, if $\overline{X\setminus V_g}$ is simply-connected, so is $X'$. 
\end{proposition}

\medskip
\begin{remark}
Propositions \ref{T:topinv_unchaining1} and \ref{T:topinv_unchaining2} show that in general unchaining operation is not equivalent to a sequence of rational blow-downs and blow-ups. For instance, if a simply-connected manifold $X'$ can be obtained by applying $(2g+1)$-unchaining to another simply-connected manifold $X$ (and we'll see many instances in this paper), we deduce that $b^+(X)-b^+(X')$ is equal to $g(g+1)$. However, rational blow-downs and blow-ups do not change $b^+$. 
\end{remark}

By Propositions~\ref{T:topinv_unchaining1} and \ref{T:topinv_unchaining2}, the unchaining operation decreases the second Betti number of a $4$--manifold, so it is a handy operation for deriving a \emph{smaller} symplectic $4$--manifold. In addition, the unchaining might give rise to new $(-1)$-spheres in $X'$, as we will see in our examples in this paper, which we can further blow-down to obtain an even smaller manifold. The same examples will demonstrate that the symplectic Kodaira dimension is also non-increasing under the unchaining operation, i.e. $\kappa (X') \leq \kappa(X)$, which we conjecture to be true in general:

\begin{conjecture}
If  $(X', \omega')$ is obtained from the symplectic $4$--manifold $(X, \omega)$ by unchaining, then their symplectic Kodaira dimensions satisfy $\kappa(X') \leq \kappa(X)$. 
\end{conjecture}

\medskip
\begin{remark}[Surgery along even chains] \label{evenchains}
Almost every aspect of the unchaining surgery we discuss here applies likewise to surgeries along even number of chains, which correspond to a similar relation in the mapping class group, where one replaces a product of Dehn twists along an even number of chains with a single Dehn twist along the boundary of their tubular neighborhood \cite{FM_MCG}. Although in this article we only discuss the odd case for brevity, the even case similarly has many interesting applications, readily available in the existing literature: For instance, even chain monodromy substitution was conclusively  employed in Mustafa Korkmaz and the first author's reverse engineering of small positive factorizations in \cite{BaykurKorkmaz}, and the examples of genus--$2$ Lefschetz fibrations in \cite{BaykurKorkmaz} readily demonstrate that the analogue of the unchaining surgery on even chains, too, can decrease the Kodaira dimension. The exotic  rational surfaces constructed in \cite{BaykurKorkmaz} and several symplectic Calabi-Yau homotopy $\K$ and homotopy Enriques surfaces in \cite{BaykurHayano} can all be seen to be obtained by such a surgery from a symplectic surfaces of Kodaira dimension two. (By the inverse of the operation, one replaces a reducible fiber component with a Stein subdomain corresponding to the even chain). 


\noindent We note that one difference in the case of a surgery along an even chain is that the operation now swaps a Stein subdomain with a \emph{symplectic filling} given by the neighborhood of a $(-1)$--curve, which is not a Stein filling of the contact structure on the boundary supported by the obvious open book. 
 \end{remark}

\medskip
\section{New symplectic pencils via unchaining}\label{Sec:newpencils}

Here we will produce new symplectic Lefschetz pencils from a family of pencils on complex surfaces of general type, by carefully applying the unchaining operation. As we keep track of the associated monodromies, we will then look at their lifts to detect a sufficient number of $(-1)$--sections to determine the Kodaira dimension of the new symplectic $4$--manifolds, using Theorem~\ref{KodFromLF}.

\subsection{Positive factorizations for a family of holomorphic pencils} \label{hurwitz} \

Let $c_1,\ldots, c_{2g+1}, \delta_1,\delta_1'\subset \Sigma_g^2$ be simple closed curves shown in Figure~\ref{F:curves1}. 
We can obtain the closed surface $\Sigma_g$ by capping $\Pa \Sigma_g^2$ by two disks. 
In this way we regard $\Sigma_g^2$ as a subsurface of $\Sigma_g$.  
By Lemma~\ref{T:chain relation} we have: 
\begin{align}
(t_{c_1}t_{c_2}\cdots t_{c_{2g+1}})^{2g+2} & = 1  \mbox{ in }\Gamma_g\mbox{, and} \nonumber\\
(t_{c_1}t_{c_2}\cdots t_{c_{2g+1}})^{2g+2} &= t_{\delta_1}t_{\delta_1^\prime} \mbox{ in }\Gamma_g^2. \label{Eq:chain relation boundary}
\end{align}
These positive factorizations prescribe a Lefschetz fibration $(Z_g, f)$ and a pencil $(Z_g',  f')$, respectively, where $(Z_g, f)$ is obtained from $(Z'_g, f')$ by blowing-up the two base points. We easily calculate the Euler characteristic as \mbox{$e(Z_g')=4g^2+2g+4$}, and thus $\eu(Z_g)=4g^2+2g+6$.
Here the left-hand side of the relation \eqref{Eq:chain relation boundary} is obtained by lifting a braid monodromy of a non-singular projective curve of degree $2(g+1)$ in $\CP$ under the double branched covering branched at $2g+2$ points (see \cite[Corollary VIII.2.3]{MoishezonTeicher}), so $f'$ is a holomorphic map, which is the composition of the double branched covering $p:Z_g'\to \CP$ of $\CP$ branched along a degree $2(g+1)$ non-singular curve, and the linear projection from $\CP$ to  $\CL$. It is easy to see that $Z_g'$ is simply-connected. Moreover, as shown in \cite{Hitchin}, the canonical bundle $K_X$ of the covering complex surface $Z'_g$ is isomorphic to $p^\ast (H^{\otimes(g-2)})$ and $p^\ast(c_1(H))\in H^2(Z_g')$ is primitive, where $H$ is the holomorphic line bundle over $\CP$ defined by a hyperplane section. In particular the signature $\sigma(Z_g')=-2(g+1)^2+2$, and thus $\sigma(Z_g)=-2(g+1)^2$, and $Z_g'$ is spin if and only if $g$ is even. Furthermore, $Z_g'$ (and its blow-up $Z_g$) is a complex surfaces of general type since $H$ is very ample.

\medskip

While it is decidedly easier to identify the Stein subdomains we would like to perform unchaining surgeries along as subfactorizations in positive factorizations of Lefschetz pencils, it is still often the case that the desired subfactorization (just like the desired handle decomposition) only emerges after deliberate manipulations of the original monodromy. The rest of this  subsection is devoted precisely to this cause, with the sole aim of deriving a suitable positive factorization of the boundary multi-twist in $\Gamma_g^2$, for each $g\geq 3$,  which is Hurwitz equivalent the factorization \eqref{Eq:chain relation boundary}:

\begin{proposition}\label{lem1}
For $g\geq 3$, let $d_j=t_{c_{j-3}}^{-1}t_{c_{j-2}}^{-1}t_{c_{j-1}}^{-1}(c_j)$ and  $e_j=t_{c_{j-3}} t_{c_{j-2}} t_{c_{j-1}}(c_j)$, for all $j=4,5,\ldots,2g+1$, where the curves $c_j,d_j,e_j$ are shown in Figures\ref{F:curves1} and\ref{vanishingcycles}.\linebreak Also set, for a shorthand notation;
\begin{align*}
&D_g=t_{d_4} t_{d_5} \cdots t_{d_{2g+1}}& &\mathrm{and}& &E_g=t_{e_{2g+1}} \cdots t_{e_5} t_{e_4}.&
\end{align*}
Then, the following equations hold in $\Gamma_g^2$:
\begin{align*}
&t_{\delta_1} t_{\delta_1^\prime} = (t_{c_1} t_{c_2} t_{c_3})^{4g} D_g E_g (t_{c_5} t_{c_6} \cdots t_{c_{2g+1}})^{2g-2}  &(g:\mathrm{odd}) \\
&t_{\delta_1} t_{\delta_1^\prime} =  (t_{c_1} t_{c_2} t_{c_3})^{4(g-1)+2} (t_{c_3} t_{c_2} t_{c_1})^2 D_g E_g (t_{c_5} t_{c_6} \cdots t_{c_{2g+1}})^{2g-2} &(g:\mathrm{even}) 
\end{align*}
which are Hurwitz equivalent to the positive factorizations in \eqref{Eq:chain relation boundary}.
\end{proposition}
\begin{figure}[hbt]
 \centering
     \includegraphics[width=12cm]{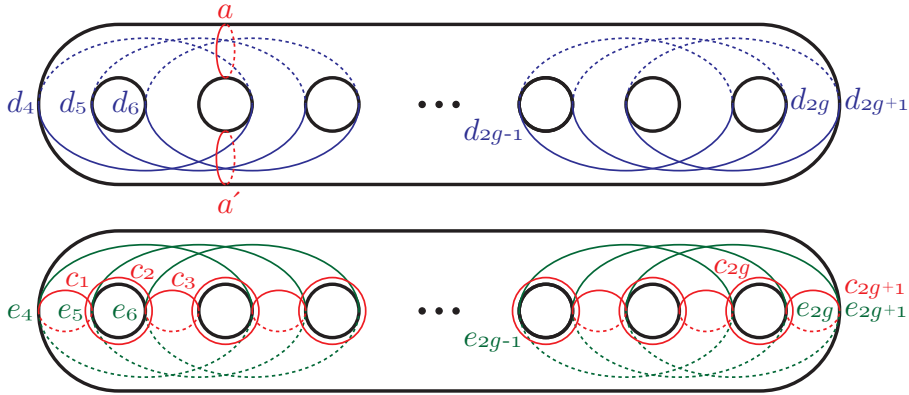}
     \caption{The curves $d_j,e_j$ on $\Sigma_g$ $(j=4,\ldots,2g+1)$.}
     \label{vanishingcycles}
\end{figure}

The proof of this proposition will require some preparation, through a sequence of technical lemmas,  Lemma~\ref{lem1-1},~\ref{lem1-2}, and \ref{lem1-3} we prove below. 

For $m=1,2,\ldots,2g-2$, and $l=m,m+1,m+2$, the braid relators amount to following Hurwitz equivalences:
\begin{align}
&t_{c_l} \cdot t_{c_{m+3}} t_{c_{m+2}} t_{c_{m+1}} t_{c_m} \sim t_{c_{m+3}} t_{c_{m+2}} t_{c_{m+1}} t_{c_m} \cdot t_{c_{l+1}}, \label{hurwitz1}\\
&t_{c_m} t_{c_{m+1}} t_{c_{m+2}} t_{c_{m+3}} \cdot t_{c_l} \sim t_{c_{l+1}} \cdot t_{c_m} t_{c_{m+1}} t_{c_{m+2}} t_{c_{m+3}}. \label{hurwitz2}
\end{align}
where, and hereon,  $\sim$ denotes {Hurwitz equivalence of positive factorizations}.

In all of the following lemmas, $g \geq 3$, and the curves $c_j, d_j$ and  $e_j$ are the ones in Figures \ref{F:curves1} and \ref{vanishingcycles}.

\begin{lemma}\label{lem1-1}
$(t_{c_1} t_{c_2} \cdots t_{c_{2g+1}})^4 \sim (t_{c_1} t_{c_2} t_{c_3})^4 (t_{c_3} t_{c_2} t_{c_1})^{2g-2} t_{d_4} t_{d_5} \cdots t_{d_{2g+1}}$ \ in $\Gamma_g^2$. 
\end{lemma}

\begin{proof}
It is easy to check using the braid relators that
\begin{align}
&(t_{c_1} t_{c_2} \cdots t_{c_{2g+1}})^4 \sim (t_{c_1} t_{c_2} t_{c_3})^4 \cdot \prod_{i=1}^4 t_{c_{5-i}} t_{c_{6-i}} \cdots t_{c_{2g+2-i}}, \label{equation2} \\
&\prod_{i=4}^1 t_{c_i} t_{c_{i+1}} \cdots t_{c_{i+2g-3}} \sim \prod_{i=1}^{2g-2} t_{c_{i+3}} t_{c_{i+2}} t_{c_{i+1}} t_{c_i}. \label{equation3}
\end{align}
Here, for $k=1,2,\ldots,2g-2$, we claim the following holds in $\Gamma_g^2$:
\begin{align*}
&\prod_{i=k}^{2g-2} t_{c_{i+3}} t_{c_{i+2}} t_{c_{i+1}} t_{c_i} \sim (t_{c_{k+2}} t_{c_{k+1}} t_{c_k})^{2g-1-k} t_{d_{k+3}} t_{d_{k+4}} \cdots t_{d_{2g+1}}. 
\end{align*}
The proof of the claim will be by induction on $2g-1-k$. For $k=2g-2$, the conclusion holds, since we have the equivalence
\begin{align*}
t_{c_{2g+1}} \cdot t_{c_{2g}} t_{c_{2g-1}} t_{c_{2g-2}} \sim t_{c_{2g}} t_{c_{2g-1}} t_{c_{2g-2}} \cdot t_{d_{2g+1}}.
\end{align*}
Assume that the relation holds for $k+1<2g-2$. 
By (\ref{hurwitz1}), we have 
\begin{align*}
&\prod_{i=k}^{2g-2} t_{c_{i+3}} t_{c_{i+2}} t_{c_{i+1}} t_{c_i} = t_{c_{k+3}} t_{c_{k+2}} t_{c_{k+1}} t_{c_k} \cdot \prod_{i=k+1}^{2g-2} t_{c_{i+3}} t_{c_{i+2}} t_{c_{i+1}} t_{c_i} \\
&= t_{c_{k+3}} t_{c_{k+2}} t_{c_{k+1}} t_{c_k} \cdot (t_{c_{k+3}} t_{c_{k+2}} t_{c_{k+1}})^{2g-1-(k+1)} t_{d_{k+4}} t_{d_{k+5}} \cdots t_{d_{2g+1}} \\
&= (t_{c_{k+2}} t_{c_{k+1}} t_{c_k})^{2g-1-(k+1)} \cdot t_{c_{k+3}} t_{c_{k+2}} t_{c_{k+1}} t_{c_k} \cdot t_{d_{k+4}} t_{d_{k+5}} \cdots t_{d_{2g+1}} \\
&= (t_{c_{k+2}} t_{c_{k+1}} t_{c_k})^{2g-1-(k+1)} \cdot t_{c_{k+2}} t_{c_{k+1}} t_{c_k} \cdot t_{d_{k+3}} \cdot t_{d_{k+4}} t_{d_{k+5}} \cdots t_{d_{2g+1}},
\end{align*}
which completes the proof of the claim. 

The proof of the lemma then follows from  $k=1$ case of the claim, and the relations (\ref{equation2}) and (\ref{equation3}). 
\end{proof}

\begin{lemma}\label{lem1-2}
The following holds in $\Gamma_g^2$: 
\begin{align*}
(t_{c_1} t_{c_2} \cdots t_{c_{2g+1}})^{2g-2} \sim t_{e_{2g+1}} \cdots t_{e_5} t_{e_4} (t_{c_1} t_{c_2} t_{c_3})^{2g-2} (t_{c_5} t_{c_6} \cdots t_{c_{2g+1}})^{2g-2}.
\end{align*}
\end{lemma}

\begin{proof}
Once again, it easily follows from the  braid relators that
\begin{align}
&(t_{c_1} t_{c_2} \cdots t_{c_{2g+1}})^{2g-2} \sim \left(\prod_{i=2g-2}^1 t_{c_i} t_{c_{i+1}} t_{c_{i+2}} t_{c_{i+3}} \right) (t_{c_5} t_{c_6} \cdots t_{c_{2g+1}})^{2g-2}. \label{equation5}
\end{align}
Here, in $\Gamma_g^2$, we claim that 
\begin{align*}
&\prod_{i=2g-2}^k t_{c_i} t_{c_{i+1}} t_{c_{i+2}} t_{c_{i+3}} = t_{e_{2g+1}} \cdots t_{e_{k+4}} t_{e_{k+3}} (t_{c_k} t_{c_{k+1}} t_{c_{k+2}})^{2g-1-k}, \label{equationb} 
\end{align*}
for $k=1,2,\ldots,2g-2$. Once again we  induct on $2g-1-k$. For $k=2g-2$ we have 
\begin{align*}
t_{c_{2g-2}} t_{c_{2g-1}} t_{c_{2g}} \cdot t_{c_{2g+1}} = t_{e_{2g+1}} \cdot t_{c_{2g-2}} t_{c_{2g-1}} t_{c_{2g}}. 
\end{align*}
Assume that the relation holds for $k+1<2g-2$. 
By (\ref{hurwitz2}), we have 
\begin{align*}
&\prod_{i=2g-2}^k t_{c_i} t_{c_{i+1}} t_{c_{i+2}} t_{c_{i+3}} = \left(\prod_{i=2g-2}^{k+1} t_{c_i} t_{c_{i+1}} t_{c_{i+2}} t_{c_{i+3}}\right) \cdot t_{c_k} t_{c_{k+1}} t_{c_{k+2}} t_{c_{k+3}} \\
& = t_{e_{2g+1}} \cdots t_{e_{k+5}} t_{e_{k+4}} (t_{c_{k+1}} t_{c_{k+2}} t_{c_{k+3}})^{2g-2-k} \cdot t_{c_k} t_{c_{k+1}} t_{c_{k+2}} t_{c_{k+3}} \\
& = t_{e_{2g+1}} \cdots t_{e_{k+5}} t_{e_{k+4}} \cdot t_{c_k} t_{c_{k+1}} t_{c_{k+2}} t_{c_{k+3}} \cdot (t_{c_k} t_{c_{k+1}} t_{c_{k+2}})^{2g-2-k} \\
& = t_{e_{2g+1}} \cdots t_{e_{k+5}} t_{e_{k+4}} \cdot t_{e_{k+3}} \cdot t_{c_k} t_{c_{k+1}} t_{c_{k+2}} \cdot (t_{c_k} t_{c_{k+1}} t_{c_{k+2}})^{2g-2-k}. 
\end{align*}
Hence, the claim is proved. 

The proof of the lemma then follows from  the $k=1$ case of the claim, and  the relation (\ref{equation5}).
\end{proof}

\begin{lemma}\label{lem1-3}
$(t_{c_3}t_{c_2}t_{c_1})^4 \sim (t_{c_1}t_{c_2}t_{c_3})^4$ \ in $\Gamma_g^2$. 
\end{lemma}
\begin{proof}
Since $t_{c_1}t_{c_3}\sim t_{c_3}t_{c_1}$, we have 
\begin{align*}
t_{c_3}t_{c_2}t_{c_1}t_{c_3}t_{c_2}t_{c_1}t_{c_3}t_{c_2}t_{c_1}t_{c_3}t_{c_2}t_{c_1} &\sim t_{c_3}t_{c_2}t_{c_3}t_{c_1}t_{c_2}t_{c_3}t_{c_1}t_{c_2}t_{c_3}t_{c_1}t_{c_2}t_{c_1}. 
\end{align*}
Here, we have the following relations, again by the braid relators:
\begin{align*}
t_{c_1}t_{c_2}t_{c_3} \cdot t_{c_i} &\sim t_{c_{i+1}} \cdot t_{c_1}t_{c_2}t_{c_3}, \\
t_{c_i}t_{c_{i+1}}t_{c_i} &\sim t_{c_{i+1}}t_{c_i}t_{c_{i+1}} \ \text{for $i=1,2$}. 
\end{align*}
Using these relations, we have \begin{align*}
(t_{c_3}t_{c_2}t_{c_3}) t_{c_1}t_{c_2}t_{c_3}t_{c_1}t_{c_2}t_{c_3} (t_{c_1}t_{c_2}t_{c_1}) &\sim t_{c_1}t_{c_2}t_{c_3} (t_{c_2}t_{c_1}t_{c_2})( t_{c_2}t_{c_3}t_{c_2}) t_{c_1}t_{c_2}t_{c_3} \\
&\sim t_{c_1}t_{c_2}t_{c_3} (t_{c_1}t_{c_2}t_{c_1})(t_{c_3}t_{c_2}t_{c_3}) t_{c_1}t_{c_2}t_{c_3} \\
&\sim t_{c_1}t_{c_2}t_{c_3} (t_{c_1}t_{c_2}t_{c_3} t_{c_1}t_{c_2}t_{c_3}) t_{c_1}t_{c_2}t_{c_3}. 
\end{align*}
Therefore, we obtain $(t_{c_3}t_{c_2}t_{c_1})^4 \sim (t_{c_1}t_{c_2}t_{c_3})^4$. 
\end{proof}

We can now give our proof of the main result of this subsection: 
\begin{proof}[Proof of Proposition~\ref{lem1}]
Applying Lemmas~\ref{lem1-1} and~\ref{lem1-2} to the odd chain relation (Lemma~\ref{T:chain relation}) we get:  
\begin{align*}
t_{\delta_1}t_{\delta_1^\prime} &= (t_{c_1} t_{c_2} \cdots t_{c_{2g+1}})^{2g+2} \\
&= (t_{c_1} t_{c_2} t_{c_3})^4 (t_{c_3} t_{c_2} t_{c_1})^{2g-2} D_g E_g (t_{c_1} t_{c_2} t_{c_3})^{2g-2} (t_{c_5} t_{c_6} \cdots t_{c_{2g+1}})^{2g-2}. 
\end{align*}
Here, note that $2g-2=4k$ if $g=2k+1$ and $2g-2=4k+2$ if $g=2k+2$. 
Since $c_1,c_2,c_3$ are disjoint from $c_5,c_6,\ldots,c_{2g+1},\delta_1,\delta_1^\prime$, by Hurwitz moves (including cyclic permutation) and Lemma~\ref{lem1-3}, we obtain the claim. 
\end{proof}

\medskip

\subsection{Families of new symplectic pencils via unchaining}\label{breaking} \

The new positive factorizations for Lefschetz pencils on $Z'_g$ we obtained in Proposition~\ref{lem1} allows us to apply unchaining as a monodromy substitution, in fact, multiple times.  In this way, we will derive new positive factorizations of the boundary multi-twist in $\Gamma_g^2$, which can be further lifted to $\Gamma_g^m$ for varying $m>2$ as we move forward. The main outcomes of this subsection will be our construction of family of symplectic fibrations $(X_g(i), f_g(i))$ and pencils $(X'_g(i), f'_g(i))$ in Theorem~\ref{thm:3}.

\begin{lemma}\label{lem2}
For $g\geq 3$, the following relations hold in $\Gamma_g^2$:
\begin{equation}
t_{\delta_1} t_{\delta_1^\prime} = \begin{cases}
t_{b} t_a^i t_{b^\prime} t_{a^\prime}^i (t_{c_1} t_{c_2} t_{c_3})^{4(g-i)} D_g E_g & (g:\mbox{odd}) \\
t_{b} t_a^i t_{b^\prime} t_{a^\prime}^i (t_{c_1} t_{c_2} t_{c_3})^{4(g-1-i)+2} (t_{c_3} t_{c_2} t_{c_1})^2 D_g E_g & (g:\mbox{even}), 
\end{cases}\label{Eq:modified chain}
\end{equation}
where $0\leq i\leq g$ if $g$ is odd and $0\leq i \leq g-1$ if $g$ is even. 
\end{lemma}

\begin{proof}
Take the following chain relators:
\begin{align*}
A_3 &:= (t_{c_1} t_{c_2} t_{c_3})^4 t_{a^\prime}^{-1} t_a^{-1}, \\
A_{2g-3} &:= (t_{c_5} t_{c_6} \cdots t_{c_{2g+1}})^{2g-2} t_{b^\prime}^{-1} t_b^{-1}.
\end{align*}
Suppose $g$ is even. By applying $A_3^{-1}$ substitutions $i$ times and $A_{2g-3}^{-1}$ substitution once to $(t_{c_1} t_{c_2} t_{c_3})^{4(g-1)}$ and $(t_{c_5} t_{c_6} \cdots t_{c_{2g+1}})^{2g-2}$  on the right-hand side of our relation in Proposition~\ref{lem1}, respectively, we obtain
\begin{align*}
t_{\delta_1} t_{\delta_1^\prime} = &  t_a^i t_{a^\prime}^i (t_{c_1} t_{c_2} t_{c_3})^{4(g-1-i)+2} (t_{c_3} t_{c_2} t_{c_1})^2 D_g E_g t_b t_{b^\prime}. 
\end{align*}
Since $a,a^\prime,b,b^\prime$ are disjoint from each other, the elements $t_{a}, t_{a'},  t_{b}$ and $t_{b'}$ all commute.  Moreover, $t_{\delta_1}t_{\delta_1^\prime}$ is central  in $\Gamma_g^2$, so we can get the claimed relation after conjugating both sides with the inverse of  $t_bt_{b^\prime}$. 

The proof for the case of odd $g$ is very similar. 
\end{proof}

We will now show that there is a lift of the relations in Lemma~\ref{lem2} from $\Gamma_g^2$ to $\Gamma_g^{2(i+1)}$ as a factorization of the boundary multi-twist. These lifts play a crucial role in our calculation of the symplectic Kodaira dimension of the underlying manifolds.

\begin{theorem}\label{thm:3}
There are symplectic genus-$g$ Lefschetz pencils $(X'_g(i), f'_g(i))$, for each $0\leq i \leq g$ if $g$ is odd, and for $0\leq i \leq g-1$ if $g$ is even, with monodromy factorizations in $\Gamma_g^{2(i+1)}$ as follows:
\begin{align*}
& t_{\delta_{i+1}} \cdots t_{\delta_2} t_{\delta_1} t_{\delta_{i+1}^\prime} \cdots t_{\delta_2^\prime} t_{\delta_1^\prime}  \\
=& \begin{cases}
t_{x_{i+1}} \cdots t_{x_2} t_{x_1} t_{x_{i+1}^\prime} \cdots t_{x_2^\prime} t_{x_1^\prime} (t_{c_1} t_{c_2} t_{c_3})^{4(g-i)} D_g E_g & (g:\mbox{odd}) \\
t_{x_{i+1}} \cdots t_{x_2} t_{x_1} t_{x_{i+1}^\prime} \cdots t_{x_2^\prime} t_{x_1^\prime} (t_{c_1} t_{c_2} t_{c_3})^{4(g-1-i)+2} (t_{c_3} t_{c_2} t_{c_1})^2 D_g E_g & (g:\mbox{even})
\end{cases}
\end{align*}
where the curves $x_k,x_k^\prime$ (for $k=1,2,\ldots, n$) on the surface $\Sigma_g^{2n}$, $n=i+1$, are as in  \mbox{Figure\ref{sectioncurves}.}
Blowing up all the base points of each $(X'_g(i), f'_g(i))$ yields symplectic Lefschetz fibrations $(X_g(i), f_g(i))$, whose monodromy factorizations in $\Gamma_g$ are obtained from the above after capping off all the boundary components.
\end{theorem}

\begin{figure}[hbt]
 \centering
     \includegraphics[width=12cm]{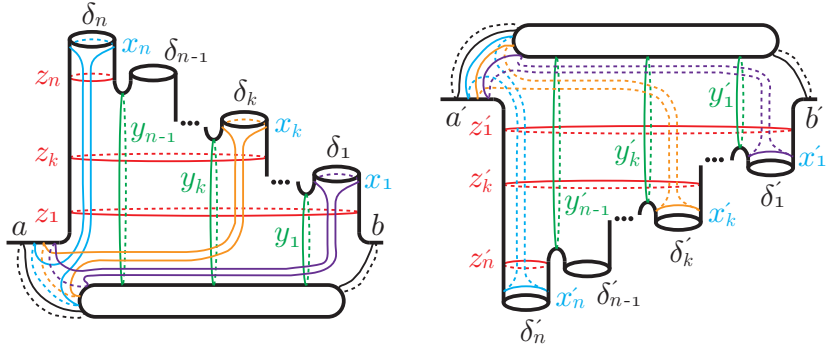}
     \caption{The curves $x_k,x_k^\prime,z_k,z_k^\prime$ and $y_k,y_k^\prime$, which are contained in pairs of pants bounded by $a, b, \delta_1$ and $a', b', \delta'_1$ on $\Sigma_g^2$ in Figure\ref{F:curves1}, respectively.}
     \label{sectioncurves}
\end{figure}

\begin{proof} 

Let $y_k,y^\prime_k$ (for $k=1,2,\ldots, n$) and $z_k,z^\prime_k$ (for $k=1,2,\ldots,n-1$) be the auxiliary curves in Figure~\ref{sectioncurves}, where $n=i+1$. Note that $y_{k}=x_{k+1}, y_{k}^\prime=x^\prime_{k+1}$ and $z_{k+1}=\delta_{k+1}, z_{k+1}^\prime=\delta_{k+1}^\prime$. 

We claim that whenever we have  a relation  in $\Gamma_g^{2k}$ of the form
\begin{align*}
t_{z_1} t_{z^\prime_1} = t_b t_a^{k-1} t_{b^\prime} t_{a^\prime}^{k-1} \cdot T,
\end{align*}
where $T$ is an element in $\Gamma_g^{2k}$, then the following relation holds in $\Gamma_g^{2k}$ as well:
\begin{align*}
t_{\delta_k} \cdots t_{\delta_2} t_{\delta_1} \cdot t_{\delta_k^\prime} \cdots t_{\delta_2^\prime} t_{\delta_1^\prime} = t_{x_k} \cdots t_{x_2} t_{x_1} \cdot t_{x_k^\prime} \cdots t_{x_2^\prime} t_{x_1^\prime} \cdot T. 
\end{align*}

We will prove the claim by induction on $k$. For $k=1$, note that $x_1=b$, $x_1^\prime=b^\prime$ and $\delta_1=z_1$, $\delta_1^\prime=z_1^\prime$. By Lemma~\ref{lem2}, this relation holds for $k=1$. Assume that the relation holds for each $1<k-1$. By the assumption, since $a,a^\prime,b,b^\prime$ are disjoint from each other, the following relation holds in $\Gamma_g^{2k}$:
\begin{align*}
t_{z_1} t_{z^\prime_1} = t_b t_a^{k-1} t_{b^\prime} t_{a^\prime}^{k-1} = t_a t_{a^\prime} \cdot t_b t_a^{k-2} t_{b^\prime} t_{a^\prime}^{k-2} \cdot T. 
\end{align*}
Therefore, since $a,a^\prime$ are disjoint from $x_i,x_i^\prime,y_{k-2},y_{k-2}^\prime$, we have the following relation in $\Gamma_g^{2k}$ by the induction hypothesis:
\begin{align*}
t_{z_{k-1}} t_{\delta_{k-2}} \cdots t_{\delta_1} \cdot t_{z^\prime_{k-1}} t_{\delta^\prime_{k-2}} \cdots t_{\delta^\prime_1} 
&= t_a t_{a^\prime} \cdot t_{y_{k-2}} t_{x_{k-2}} \cdots t_{x_1} \cdot t_{y^\prime_{k-2}} t_{x^\prime_{k-2}} \cdots t_{x^\prime_1} \cdot T \\
&= t_a t_{y_{k-2}} t_{x_{k-2}} \cdots t_{x_1} \cdot t_{a^\prime} t_{y^\prime_{k-2}} t_{x^\prime_{k-2}} \cdots t_{x^\prime_1} \cdot T. 
\end{align*}
Then, we multiply both sides by $t_{\delta_k} t_{\delta_{k-1}} t_{\delta_k^\prime} t_{\delta_{k-1}^\prime}$. Since $t_{\delta_{k-1}}$, $t_{\delta_k}$, $t_{\delta_{k-1}^\prime}$ and $t_{\delta_k^\prime}$ are central elements in $\Gamma_g^{2k}$, we obtain:
\begin{align*}
&t_{z_{k-1}} \cdot t_{\delta_k} t_{\delta_{k-1}} \cdot t_{\delta_{k-2}} \cdots t_{\delta_1} \cdot  t_{z_{k-1}^\prime} \cdot t_{\delta_k^\prime} t_{\delta_{k-1}^\prime} \cdot t_{\delta_{k-2}^\prime} \cdots t_{\delta_1^\prime} \\
&=
t_{\delta_k} t_a t_{y_{k-2}} t_{\delta_{k-1}} \cdot t_{x_{k-2}} \cdots t_{x_1} \cdot t_{\delta_k^\prime} t_{a^\prime} t_{y^\prime_{k-2}} t_{\delta_{k-1}^\prime} \cdot t_{x^\prime_{k-2}} \cdots t_{x^\prime_1} \cdot T. 
\end{align*}
By the lantern relations $t_{\delta_k} t_a t_{y_{k-2}} t_a t_{\delta_{k-1}} = t_{z_{k-1}} t_{y_{k-1}} t_{x_{k-1}}$ and $t_{\delta_k^\prime} t_{a^\prime} t_{y^\prime_{k-2}} t_{\delta_{k-1}^\prime} = t_{z^\prime_{k-1}} t_{y^\prime_{k-1}} t_{x^\prime_{k-1}}$, we have 
\begin{align*}
&t_{z_{k-1}} t_{\delta_k} t_{\delta_{k-1}} \cdot t_{\delta_{k-2}} \cdots t_{\delta_1}  \cdot t_{z^\prime_{k-1}} t_{\delta_k^\prime} t_{\delta_{k-1}^\prime} \cdot t_{\delta_{k-2}^\prime} \cdots t_{\delta_1^\prime} \\
&=  t_{z_{k-1}} t_{y_{k-1}}  t_{x_{k-1}} t_{x_{k-2}} \cdots t_{x_1} \cdot  t_{z^\prime_{k-1}} t_{y^\prime_{k-1}} t_{x_{k-1}^\prime} t_{x_{k-2}^\prime} \cdots t_{x^\prime_1} \cdot T. 
\end{align*}
Because $z_{k-1}^\prime$ are disjoint from $x_i$, $y_i$, we can remove $t_{z_{k-1}}$ and $t_{z^\prime_{k-1}}$ from both sides of this relation. 
Since $y_{k-1}=x_k$, $y_{k-1}^\prime=x^\prime_k$, this concludes the proof of the claim.

The proof of the theorem now follows from applying the above lifting argument to the positive factorization given in Lemma~\ref{lem2}. Corresponding to it is the promised family of symplectic Lefschetz pencils $(X'_g(i), f'_g(i))$. Capping all the boundary components induces a well-known homomorphism $\Gamma_g^{2(i+1)} \to \Gamma_g$, under which the monodromy factorization of this pencil maps to a monodromy factorization of a symplectic Lefschetz fibration $(X_g(i), f_g(i))$, which clearly realizes the blow-up of the pencil at all its base points. 
\end{proof}

\medskip
\subsection{Topology of the symplectic $4$--manifolds $X_g(i)$ and $X_g'(i)$} \

Here we will investigate the topology of the symplectic manifolds  $X'_g(i)$ and therefore that of $X_g(i) \cong X_g(i) \, \# \, 2(i+1) \CPb$, which are the total spaces of the new  pencils and fibrations we produced above.  We will record all the results for $X'_g(i)$, which then translate to those of $X_g(i)$ through blow-ups.

\begin{lemma}\label{topXg}
For $0\leq i \leq g$ when $g$ is odd, and $0\leq i \leq g-1$ when $g$ is even,
\begin{align*}
&e(X_g'(i))=12(g-i)  \hspace{.7em}\mbox{and}\hspace{.7em} \sigma(X_g'(i))=-8(g-i).
\end{align*}
When $i\leq g-1$, $X_g(i)$ is simply-connected, whereas for $i=g$ (when $g$ is odd), we have $\pi_1(X_g(g))\cong \Z\oplus \Z$.
\end{lemma}

\begin{proof}
The positive factorization  $t_{\delta_1}t_{\delta_1^\prime}$ in Theorem~\ref{thm:3}, which gives rise to $X_g(i)$ and $X_g'(i)$, is obtained by  applying $A_3$-substitutions $i$ times and $A_{2g-3}$-substitution once to the factorization of $t_{\delta_1}t_{\delta_1^\prime}$ in Proposition~\ref{lem1}. (Note that $X_g(0) \neq Z_g$ since we still have the latter substitution in effect.) It means that for $g \geq 3$, $X_g(i)$ is obtained by applying $3$-unchaining $i$ times and $(2g-3)$-unchaining once to $Z_g$, and then $X_g'(i)$ is obtained by blowing-down $X_g(i)$ $2(i+1)$-times. The Euler characteristics and signature calculations then follow directly from those of $Z_g$ (Subsection~\ref{hurwitz}), and Proposition~\ref{T:topinv_unchaining1}.

To calculate $\pi_1(X_g(i))$ we can invoke Proposition~\ref{T:topinv_unchaining2}, or rely on the handle decomposition induced by the global Lefschetz fibration to carry out a more direct calculation as follows: Let $a_1,b_1,\ldots,a_g,b_g$ be generators of $\pi_1(\Sigma_g)$ as shown in Figure~\ref{F:generator_pi1} (here we regard $\Sigma_g^2$ as a subsurface of $\Sigma_g$). Set $\gamma_k=[a_1,b_1]\cdots[a_k,b_k]$, which is represented by a separating curve, and $a_0=b_0=a_{g+1}=\gamma_{g+1}=\gamma_0=1$ in $\pi_1(\Sigma_g)$. 
Up to conjugation in  $\pi_1(\Sigma_g)$, we have: 
\begin{align*}
&c_1 = a_1, \ c_2 = b_1, \ c_3 = b_1a_1^{-1}b_1^{-1}a_2, \ a = a_2, \ a^\prime =\gamma_1a_2, \\
&d_{2k} = b_k a_k^{-1}b_{k-1}b_{k-2}a_{k-2}b_{k-2}^{-1}, \ 2\leq k \leq g, \\
&d_{2k+1} = a_{k+1}^{-1}b_kb_{k-1}a_{k-1}, \ 2\leq k \leq g, \\
&e_{2k} = \gamma_ka_kb_kb_{k-1}a_{k-2}^{-1}\gamma_{k-2}^{-1}, \ 2\leq k\leq g, \\
&e_{2k+1} = \gamma_{k+1}a_{k+1}b_kb_{k-1}a_{k-1}^{-1}\gamma_{k-1}^{-1}, \ 2\leq k\leq g 
\end{align*}
Recall that $\pi_1(X) \cong  \pi_1(\Sigma_g) \, / \, N$, where $N$ is the subgroup of $\pi_1(\Sigma_g)$ normally generated by the vanishing cycles of the Lefschetz fibration $(X_g(i), f_g(i))$, which are the Dehn twist curves in its monodromy factorization. It is then easy to see that the collection of Dehn twist curves in the factorization of kill all the generators through the above identities. So $\pi_1(X_g(i)) =1$ for $i=0,1,2,\ldots,g-1$. 

When $g$ is odd, we no longer  have the Dehn twist curves $c_1, c_2, c_3$, and without the implied relations  $c_1=c_2=c_3=1$ (yielding $a_1=b_1=a_2=1$, and in turn $b_2=1$), the remaining relations give  $\pi_1(X_g(g))\cong\langle a_1,b_1 \mid \gamma_1=1 \rangle = \mathbb{Z}\oplus\mathbb{Z}$. 
\end{proof}

Finally, armed with our deeper knowledge of the presence of ---sufficiently many--- exceptional sections in each Lefschetz fibration $(X_g(i), f_g(i))$, we can utilize Theorem~\ref{KodFromLF} to determine the symplectic Kodaira dimension of majority of these manifolds:

\begin{theorem}[Symplectic Kodaira dimension of $X'_g(i)$] \label{T:Kodaira dim X_g(i)} 
Let $g \geq 3$, $0\leq i \leq g$ when $g$ is odd, and $0\leq i \leq g-1$ when $g$ is even.
In all of the following cases, $X'_g(i)$ is a symplectic $4$--manifold with Kodaira dimension
\[
	\kappa(X'_g(i)) = \begin{cases}
	-\infty &  \textit{if}  \ \  i \geq g - 1 \\
	\ \ 0 & \ \ \ i= g-2 \\
	\ \ 1 & \ \ \  i=g-2n  \ \ \text{or} \ i=g-3, 
	\end{cases}
	\] 
In all of the cases above with $i \leq  g- 2$, $X'_g(i)$ is spin, and thus minimal. Moreover,  $X_g'(g-1)\cong E(1)$ and  $X_g'(g) \cong S^2\times T^2$ when $g$ is odd. 
\end{theorem}

\begin{proof}
First of all, let us observe that, since the Euler characteristic and signature satisfy $e = 2 - 2b_1 + b_2^+ + b_2^-$ and $\sigma = b_2^+ - b_2^-$, our $\pi_1$ calculations in Theorem~\ref{topXg} imply that for $g\geq3$, and $i=0,1,2,\ldots,g-2$, $b_1(X_g'(i))=0$, $b_2^+(X_g'(i)) = 2(g-i)-1$ and $b_2^-(X_g'(i))=10(g-i)-1$. So  $X_g'(i)$ (and $X_g(i)$) is neither a rational nor a ruled surface when  $i=0,1,2,\ldots.g-2$. 

With this observation, the cases of $\kappa(X'_g(i))=\infty$ and $0$ are now immediate from Theorem~\ref{KodFromLF}. It remains to prove that when $i=g-2n$ or $i=g-3$,  $\kappa(X'_g(i)=1$. 

By Lemma~\ref{T:spin X_g'(i)} in the next section, $X_g'(g-2n)$ is spin, and in particular, minimal. 
We can further deduce from the Euler characteristic and signature calculations in Lemma~\ref{topXg} that $K_{X_g'(g-2n)}^2= 2\eu(X_g'(g-2n)+3 \sigma(X_g'(g-2n))=0$.
Moreover, $b_2^+(X_g'(g-2n)) > 3$ by the above calculation, so it cannot have Kodaira dimension zero by \cite{Li4} (see also \cite{Bauer}). It follows that $\kappa(X_g'(g-2n))=1$.

By the above calculations again, $X_g^\prime(g-3)$ is not rational and ruled, and 
$b_2^+>3$. So it cannot have Kodaira dimension $-\infty$ or $0$. If it's Kodaira dimension is $2$, then because the Lefschetz fibration $f_{g}(g-3)$ has $2g-4$ sections, we deduce from a generalization of Theorem~\ref{KodFromLF} (see \cite{Sato, BaykurHayano}) that there are no other exceptional sections (which would necessarily be multisections).  So $X_g'(g-3)$ would be minimal. On the other hand, by Lemma~\ref{topXg}, we have $K_{X_g^\prime(g-3)}^2=2\eu(X_g^\prime(g-3))+3\sigma(X_g^\prime(g-3))=0$. \linebreak
This contradicts that $K^2>0$ for a minimal symplectic $4$--manifold with $\kappa=2$. 

Our discussion here shows that all these manifolds with $\kappa = 0$ or $1$ are spin.

For the diffeomorphism statements, recall that when $\kappa=-\infty$ we either have a rational or ruled surface in hand. So from our calculation of the topological invariants  in Lemma~\ref{topXg}, we deduce that   $X'_g(g-1) \cong E(1)$, and $X_g'(g)$ is diffeomorphic to either $S^2\times T^2$ or $S^2\tilde{\times} T^2$. As we'll see that $X_g'(g)$ is spin in  Theorem~\ref{T:spin X_g'(i)}), we conclude that $X_g'(g) \cong S^2\times T^2$.
\end{proof}

\begin{remark}
For the values of $i$ that are not covered in the theorem, we expect that $\kappa(X'_g(i))=1$, which would be immediate once these manifolds are seen to be minimal.
\end{remark}

Among the families we have, the pencils $(K_g, k_g):=(X'_g(g-2), f'_g(g-2))$ are perhaps the most interesting. It follows from our calculations of the topological invariants and by Freedman's theorem that each $K_g$ is a symplectic Calabi-Yau surface homeomorphic to the $\K$ surface. Moreover, they have the same Seiberg-Witten invariants as the $\K$ surface (which is a formal property for such an SCY).

\begin{corollary}[Pencils on symplectic Calabi-Yau homotopy $\K$ surfaces] \label{SCYK3s}
For each $g \geq 3$,  there is a symplectic genus-$g$ Lefschetz pencil $(K_g, k_g)$, where $K_g$ is a symplectic Calabi-Yau homotopy $\K$ surface, and $k_g$ has the monodromy factorization: 
\begin{align*}
& t_{\delta_{g-1}} \cdots t_{\delta_2} t_{\delta_1} t_{\delta_{g-1}^\prime} \cdots t_{\delta_2^\prime} t_{\delta_1^\prime}  \\
=& \begin{cases}
t_{x_{g-1}} \cdots t_{x_1} t_{x_{g}^\prime} \cdots  t_{x_1^\prime} (t_{c_1} t_{c_2} t_{c_3})^{8} t_{d_4}  \cdots t_{d_{2g+1}} t_{e_{2g+1}} \cdots  t_{e_4} & (g:\mbox{odd}) \\
t_{x_{g-1}} \cdots t_{x_1} t_{x_{g}^\prime} \cdots  t_{x_1^\prime} (t_{c_1} t_{c_2} t_{c_3})^{6} (t_{c_3} t_{c_2} t_{c_1})^2 t_{d_4} \cdots t_{d_{2g+1}} t_{e_{2g+1}} \cdots  t_{e_4} & (g:\mbox{even})
\end{cases}
\end{align*}
in  $\Gamma_g^{2g-2}$, where the curves $\delta_j, \delta'_j, c_j, d_j, e_j, x_j,x'_j$ are as in  Figures~\ref{F:curves1}\ref{vanishingcycles}\ref{sectioncurves}. 
\end{corollary}

A natural question, which might be tractable  through the explicit handle diagrams prescribed by the factorizations above is: 

\begin{question}
Are the symplectic Calabi-Yau homotopy $\K$ surfaces $K_g$,  $g \geq 3$, all diffeomorphic to the standard $\K$ surface?
\end{question}

\medskip
\begin{remark} \label{contracting}
Observe that, we can view $K_g$ to be obtained from the rational surface $E(1)$ by the natural inverse of the unchaining surgery. We speculate that when this inverse operation can be performed in the complex category, it would correspond to contracting a  $(-2)$--curve (of genus $(g-2)$ for these examples) and than taking its minimal resolution. This leads to another interesting question: what are the symplectic Calabi-Yau surfaces one can get by contracting a symplectic $(-2)$--curve in a rational surface and taking its minimal symplectic resolution?
\end{remark}

\medskip
\section{Pencils on spin $4$--manifolds and Stipsicz's conjecture} \label{Sec:spin}
%

In this section, we will first extend a result of Stipsicz in \cite{Stipsicz_spin}, which gave a characterization of a Lefschetz fibration to be a spin $4$--manifold, to that of Lefschetz pencils (Theorem~\ref{T:condition spin LP}). These are given in terms of the $\Z_2$--homology classes of the vanishing cycles and exceptional sections of the associated positive factorization. We will then address Stipsicz's conjecture on the existence of $(-1)$--sections in fiber sum indecomposable Lefschetz fibrations \cite{Stipsicz_indecomposability} by providing counter-examples of any genus $g \geq 2$ (Theorem~\ref{T:counterex Stipsicz}).

\subsection{Spin structures on Lefschetz pencils} \

Here is our characterization of whether the total space of a Lefschetz pencil is spin, in terms of the associated positive factorization:

\begin{theorem}
[Spin characterization from pencil monodromies] \label{T:condition spin LP}
Let $(X,f)$ be a genus--$g$ Lefschetz pencil with a monodromy factorization \ $t_{c_1}\cdots t_{c_n}=t_{\delta_1}\cdots t_{\delta_p}$. Then $X$ admits a spin structure if and only if there exists a quadratic form $q:H_1(\Sigma_g^p;\mathbb{Z}/2\mathbb{Z})\to \mathbb{Z}/2\mathbb{Z}$ with respect to the intersection pairing of $H_1(\Sigma_g^p;\mathbb{Z}/2\mathbb{Z})$ such that $q(c_i)=1$ for any $i$ and $q(\delta_j)= 1$ for some $j$. 
\end{theorem}

\noindent
\noindent
Since we can obtain the total space of a Lefschetz \textit{pencil} by attaching $4$-handles to the complement of $(-1)$-sections of a Lefschetz \textit{fibration}, Theorem~\ref{T:condition spin LP} will  follow from the following lemma: 

\begin{lemma}\label{T:spinstructure_complement}
Let $(X, f)$ be a Lefschetz fibration, $S_1,\ldots,S_p\subset X$ be disjoint sections of $f$ and \ $t_{c_1}\cdots t_{c_n}=t_{\delta_1}^{a_1}\cdots t_{\delta_p}^{a_p}$ be the corresponding monodromy factorization. \linebreak 
For $\mathcal{S} = \sqcup_i S_i$, the disjoint union of $S_i$s, the complement $X\setminus \mathcal{S}$ admits a spin structure if and only if there exists a quadratic form $q:H_1(\Sigma_g^p;\mathbb{Z}/2\mathbb{Z})\to \mathbb{Z}/2\mathbb{Z}$ with respect to the intersection pairing of $H_1(\Sigma_g^p;\mathbb{Z}/2\mathbb{Z})$, which satisfies:
\begin{enumerate}[(A)]
\item
$q(c_i)=1$ for any $i$, 
\item
$q(\delta_j)\equiv a_j$ for some $j$. 
\end{enumerate}
\end{lemma}

\begin{proof}
The complement $X\setminus \mathcal{S}$ can be decomposed as follows: 
\begin{equation}\label{Eq:decomposition totalsp}
X\setminus \mathcal{S} = D^2\times \Sigma_g^p \cup (h_1\sqcup \cdots \sqcup h_n)\cup D^2\times \Sigma_g^p, 
\end{equation}
where $h_i$ is a $2$-handle attached along the vanishing cycle $c_i$. 
We can easily deduce that there is a one-to-one correspondence between the set of isomorphism classes of spin structures on $D^2\times\Sigma_g^p$ and the set of $\Z/2\Z$-valued quadratic functions on $H_1(\Sigma_g^p;\Z/2\Z)$ with respect to the intersection pairing. 
This correspondence assigns a spin structure $\mathfrak{s}$ of $D^2\times \Sigma_g^p$ to a quadratic function which takes the value $0$ on $c\in H_1(\Sigma_g^p;\Z/2\Z)$ if the restriction of $\mathfrak{s}$ on a circle representing $c$ can be extended to a spin structure over a $2$-manifold and takes the value $1$ otherwise (the reader can refer to \cite{Stipsicz_spin} for the case of $D^2\times \Sigma_g$).  
We denote the subset $D^2\times \Sigma_g^p \cup (h_1\sqcup \cdots \sqcup h_n)\subset X\setminus \mathcal{S}$ by $X'$. 
As is shown in \cite{Stipsicz_spin}, the condition (A) in Lemma~\ref{T:spinstructure_complement} holds if and only if the associated spin structure on $D^2\times\Sigma_g^p$ can be extended to that on $X'$. 
The latter $D^2\times \Sigma_g^p$ in the decomposition \eqref{Eq:decomposition totalsp} can be regarded as the union of a $2$-handle $D^2\times B$, where $B$ is a small ball close to the boundary component near $\delta_j$, and $2g+p-1$ $3$-handles. 
Thus, $X\setminus \mathcal{S}$ admits a spin structure if and only if there exists a spin structure on $X'$ which can be extended to that on $X'\cup (D^2\times B)$. 

In what follows, we identify a surface $\Sigma_g^p$ with a fiber in $\partial X'$. 
Let $A$ be an annulus neighborhood of the boundary component near $\delta_j$ which contains $B$ and is away from any of the vanishing cycles $c_1,\ldots,c_n$. 
We take a parallel transport self-diffeomorphism $\varphi$ of $\Sigma_g^p$ along the boundary of $f(X')$ so that it preserves $\partial \Sigma_g^p$ and $A$ point-wise. 
We can then obtain the following diffeomorphism:
\[
\partial X' \cong ([0,1]\times \Sigma_g^p)/(1,x)\sim (0,\varphi(x)). 
\]
By the assumption $\varphi$ is isotopic (relative to $A$) to the $a_j$-th power of the Dehn twist diffeomorphism along a simple closed curve near $\partial A$. 
The latter diffeomorphism is further isotopic to the identity via the isotopy (supported on a neighborhood of $A$) described in Figure~\ref{F:isotopy_rotation}. 
\begin{figure}[htbp]
\subfigure[The initial and the final configuration.]{
\includegraphics[height=35mm]{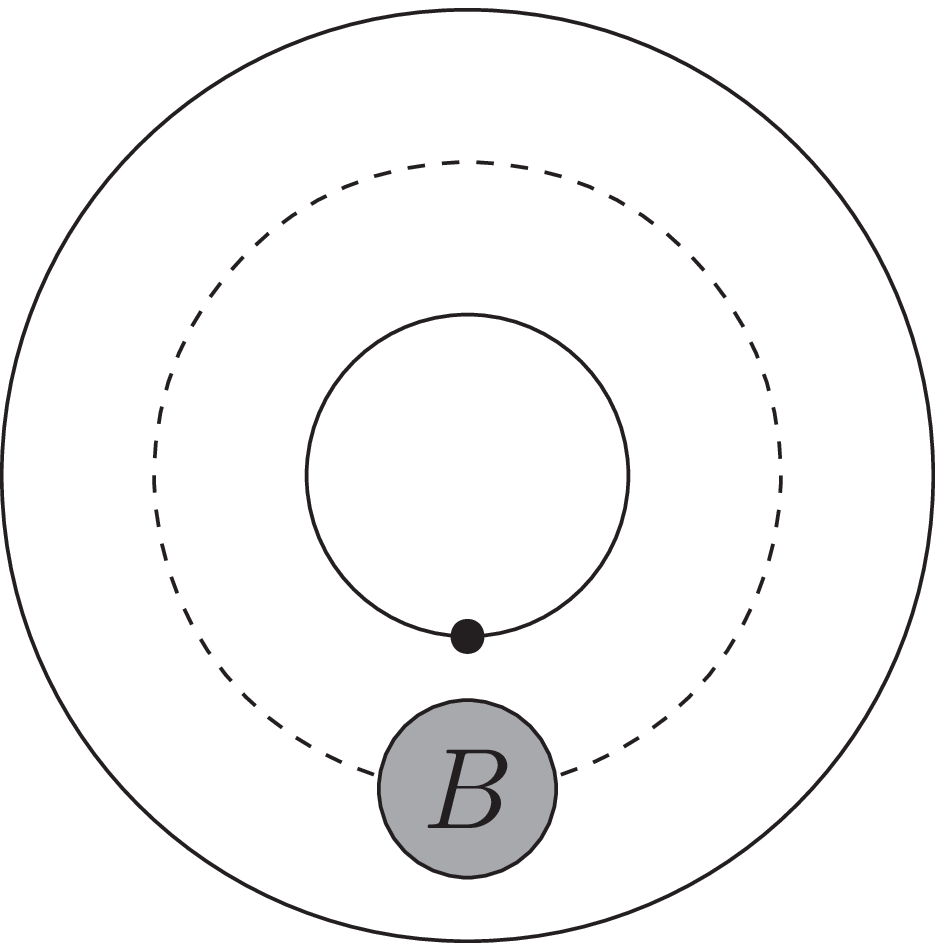}
\label{F:isotopy_rotation1}
}
\hspace{.5em}
\subfigure[As the time parameter increases, $B$ rotates along the meridian circle and the boundary component also rotates in the opposite way.]{
\includegraphics[height=35mm]{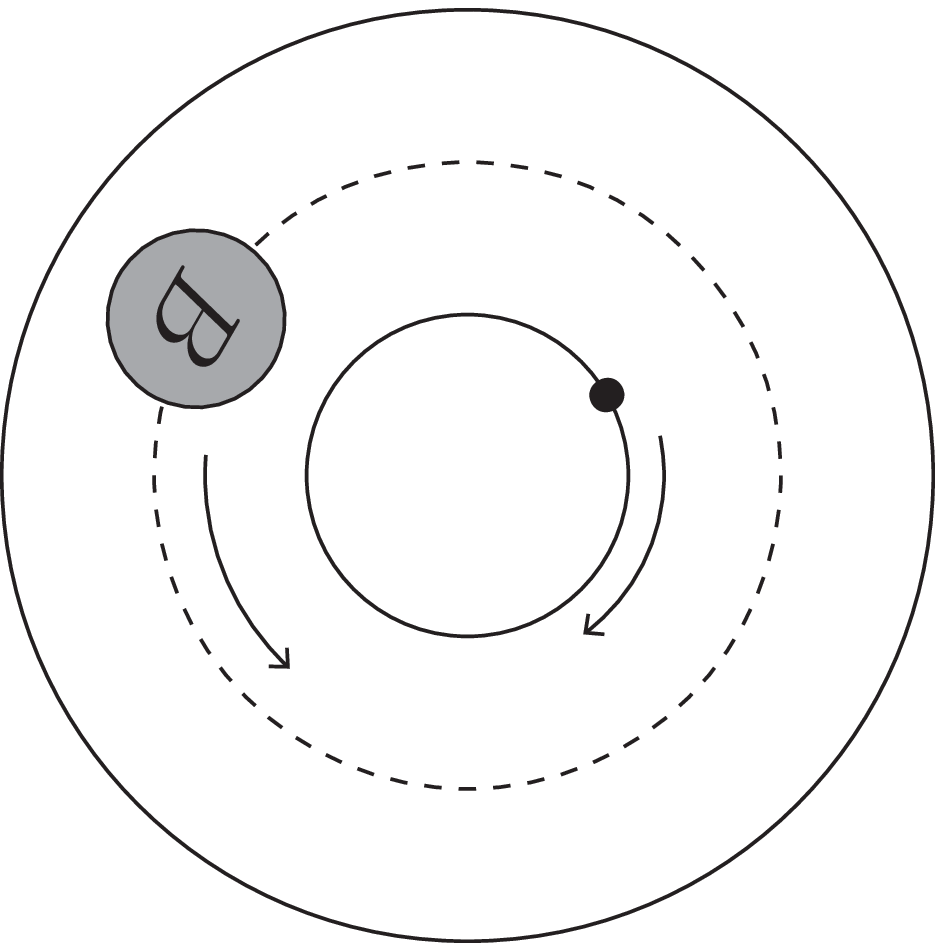}
\label{F:isotopy_rotation2}
}
\caption{The inner circles are the boundary components near $\delta_j$, while the outer circles are parallel to $\partial A$.}
\label{F:isotopy_rotation}
\end{figure}
As the figure shows, the isotopy makes $B$ rotate around the boundary component $a_j$ times keeping the inward tangent vector at the center of $B$ inward, and also makes the boundary component $-a_j$ times.  
We denote the concatenation of the two isotopies by $\varphi_t$ ($t\in [0,1]$). 
Identifying $\partial D^2\times\Sigma_g^p$ with $([0,1]\times \Sigma_g^p)/(1,x)\sim (0,x)$, we can explicitly give an attaching map of $D^2\times \Sigma_g^p$ to $X'$ as follows: 
\[
([0,1]\times \Sigma_g^p)/(1,x)\sim (0,x) \ni (t,x)\mapsto (t,\varphi_t(x))\in ([0,1]\times \Sigma_g^p)/(1,x)\sim (0,\varphi(x)). 
\]
We denote this attaching map by $\Phi$. 
The attaching map of the $2$-handle $D^2\times B$ is the restriction $\Phi|_{\partial D^2\times B}$. 
For a given $\Z/2\Z$-valued quadratic function $q$ on $H_1(\Sigma_g^p;\Z/2\Z)$, the restriction of the associated spin structure on the image \linebreak \mbox{$\Phi(D^2\times B)$} bounds a spin $2$-manifold if $q(a_j\delta_j)=0$ and does not bound otherwise. 
Moreover, it is easy to deduce from the definition of $\Phi$ that the pull-back of a spin structure $\mathfrak{s}$ on $\Phi(D^2\times B)$ by $\Phi|_{\partial D^2\times B}$ bounds a spin $2$-manifold if and only if either $a_j$ is odd and $\mathfrak{s}$ does not bound a spin $2$-manifold, or $a_j$ is even and $\mathfrak{s}$ bounds a spin $2$-manifold. 
The latter condition is equivalent to the condition (B) in Lemma~\ref{T:spinstructure_complement}. 
\end{proof}

\begin{remark}

The proof of Lemma~\ref{T:spinstructure_complement} also shows that for a quadratic form $q$ with the condition (A), the condition (B) holds if and only if $q(\delta_j)\equiv a_j$ for \textit{any} $j$. 
Furthermore, if we can find a quadratic form with the condition (A) and one of the section $S_1,\ldots,S_p$ has even self-intersection, the condition (B) is automatically satisfied. 
Indeed, $q$ induces a quadratic form on $H_1(\Sigma_g;\mathbb{Z}/2\mathbb{Z})$ and it can be extended to that on $H_1(\Sigma_g^p;\mathbb{Z}/2\mathbb{Z})$ so that the values of $\delta_i$'s are all zero. (An alternative way to deduce this observation is to find a spin structure on $X$ by applying Stipsicz's result to the induced quadratic form on $H_1(\Sigma_g;\mathbb{Z}/2\mathbb{Z})$.)

\end{remark}

\medskip
\subsection{Stipsicz's conjecture on exceptional sections} \

We will now prove that: 
\begin{theorem}[Counter-examples to Stipsicz's conjecture] \label{T:counterex Stipsicz}
For any $g\geq 3$, there exists a genus--$g$ fiber-sum indecomposable  Lefschetz fibration without any exceptional sections. 
\end{theorem}

\noindent Unlike the handful of earlier counter-examples with $g=2,3$ \cite{Sato2010, BaykurHayano, AkhmedovMonden}, these examples will have Kodaira dimension $1$ when $g \geq 4$. 

We will describe these counter-examples with explicit monodromy factorizations. The spin characterization in Theorem~\ref{T:condition spin LP} will play a vital role here  to pin down the exact number of exceptional spheres in the total spaces. 

Let $c_1,\ldots,c_{2g+1},a,b,a',b'\subset \Sigma_g^{2}$ be the simple closed curves shown in Figure~\ref{F:curves1}.  As we implicitly did in Section~\ref{Sec:newpencils}, for any $n\geq 1$ we regard $\Sigma_g^2$ as a subsurface of $\Sigma_g^{2n}$ so that the curves $a,b,a',b'$ are embedded in $\Sigma_g^{2n}$ as shown in Figure~\ref{sectioncurves}.  In what follows, we use symbols representing simple closed curves in $\Sigma_g^{2n}$ (such as $c_i,a,b$) to represent homology classes in $H_1(\Sigma_g^{2n};\Z/2\Z)$ represented by the corresponding curves. 

\begin{lemma}\label{T:existence quadratic form}
For any $g\geq 3$ and $n\geq 1$ such that $g+n$ is odd, there exists a quadratic form $q:H_1(\Sigma_g^{2n};\Z/2\Z)\to \Z/2\Z$ with respect to the intersection pairing which satisfies 
\begin{enumerate}
\item
$q(c_1)= \cdots =q(c_{2g+1})=1$, 
\item
$q(d_4)=\cdots = q(d_{2g+1})=q(e_4)=\cdots = q(e_{2g+1})=1$, 
\item
$q(x_1)=\cdots =q(x_n)=q(x_1')=\cdots =q(x_n')=1$, and
\item
$q(\delta_1)=\cdots =q(\delta_n)=q(\delta_1')=\cdots =q(\delta_n')=1$. 
\end{enumerate}
\end{lemma}
\begin{proof}
Since the elements $c_1,\ldots,c_{2g},\delta_1,\ldots,\delta_n,\delta_2',\ldots,\delta_n'$ generate $H_1(\Sigma_g^{2n};\Z/2\Z)$, there exists a quadratic form $q:H_1(\Sigma_g^{2n};\Z/2\Z)\to \Z/2\Z$ which assigns the value $1$ to all the elements in the generating set. 
The following equalities (in $H_1(\Sigma_g^{2n};\Z/2\Z)$) can be verified easily: 
{\allowdisplaybreaks
\begin{align*}
&c_{2g+1}=c_1+c_3+\cdots + c_{2g-1}+\delta_1+\delta_2+\cdots +\delta_n, \\
&d_j=e_j=c_{j-3}+c_{j-2}+c_{j-1}+c_j, \\
&x_i=c_1+c_3+\delta_i,~~x_i'=c_1+c_3+\delta_i', \\
&\delta_1'=\delta_1+\cdots +\delta_n+\delta_2'+\cdots +\delta_n'. 
\end{align*}
}
Using these equalities we can show by direct calculation that the quadratic form $q$ satisfies the desired conditions (we need the assumption on $g$ and $n$ here since $q(d_{2g+1})$ is equal to $g+n$). 
\end{proof}

Let $X_g'(i)$ be the symplectic genus--$g$ Lefschetz pencil with $2(i+1)$ base points we constructed in Section~\ref{Sec:newpencils}, which has the following monodromy factorization: 
\begin{align*}
&t_{\delta_{i+1}} \cdots t_{\delta_1} t_{\delta_{i+1}^\prime} \cdots t_{\delta_1^\prime} \\
=&\begin{cases}
t_{x_{i+1}} \cdots t_{x_1} t_{x_{i+1}^\prime} \cdots t_{x_1^\prime} (t_{c_1} t_{c_2} t_{c_3})^{4(g-i)} D_g E_g & (g\mbox{~:~odd}), \\
t_{x_{i+1}} \cdots t_{x_1} t_{x_{i+1}^\prime} \cdots t_{x_1^\prime} (t_{c_1} t_{c_2} t_{c_3})^{4(g-i)-2} (t_{c_3} t_{c_2} t_{c_1})^2 D_g E_g & (g\mbox{~:~even}). 
\end{cases}
\end{align*}
%

\begin{lemma}\label{T:spin X_g'(i)}
The manifold $X_g'(i)$ is spin if and only if $g+i$ is even. 
\end{lemma}

\begin{proof}
Since the signature of $X_g'(i)$ is equal to $-8(g-i)$ (see Lemma~\ref{topXg}), $X_g'(i)$ is not spin if $g+i$ is odd. On the other hand, when $g+i$ is even, the quadratic form obtained in Lemma~\ref{T:existence quadratic form} above satisfies the necessary conditions in Theorem~\ref{T:condition spin LP}, so $X'_g(i)$ is spin in this case.
\end{proof}

We are now ready to prove the main result of this subsection:

\begin{proof}[Proof of Theorem~\ref{T:counterex Stipsicz}]
Recall that, by Lemma~\ref{topXg} that the signature and the Euler characteristic $X_g'(i)$ are respectively equal to $\eu(X'_g(i))=12(g-i)$ and $\sigma(X'_g(i))=-8(g-i)$, and $X_g'(i)$ is simply--connected when $i<g$.  So, $b_2^+(X_g'(i))=2(g-i)-1$ if $i<g$ and $K_{X_g'(i)}^2=2\chi+3\sigma=0$, where $K_{X_g'(i)}$ is the canonical class of  $X_g'(i)$. 

Suppose that $g$ is even. The manifold $X_g'(0)$ is spin by Lemma~\ref{T:spin X_g'(i)}, and in particular minimal. Thus, $X_g(0)$ has positive Kodaira dimension since $b^+_2(X_g'(0))>3$ \linebreak (see Theorem~\ref{KodFromLF}).  Since $K_{X_g'(0)}^2=0$, we have $\kappa( X_g(0))=1$. 
The monodromy factorization of the Lefschetz fibration on $X_g(0)$ can be changed by Hurwitz moves as follows: 
{\allowdisplaybreaks
\begin{align*}
&t_{x_1} t_{x_1^\prime} (t_{c_1} t_{c_2} t_{c_3})^{4g-2} (t_{c_3} t_{c_2} t_{c_1})^2 D_g E_g \\
\sim & \underline{t_{x_1} t_{x_1^\prime} t_{c_1}^2}t_{t_{c_1}^{-1}(c_2)}t_{c_3}t_{c_2}t_{c_3} (t_{c_1} t_{c_2} t_{c_3})^{4g-4} (t_{c_3} t_{c_2} t_{c_1})^2 D_g E_g.
\end{align*}
}

The curves $x_1, x_1',\delta_1,\delta_1'$ and two disjoint curves parallel to $c_1$ bound a sphere with six boundary components. 
Thus, we can apply the \emph{braiding lantern substitution} \cite[Lemma 5.1]{BaykurHayano}) the underlined part above.  This substitution replaces two disjoint exceptional sections with an exceptional bisection.  Furthermore, by \cite[Theorem 3.1]{EndoGurtas}), this amounts to a rational blowdown of a symplectic $(-4)$--sphere which can be viewed as the union of the four holed sphere on the fiber and the four disjoint vanishing cycles in the lantern configuration, whereas by \cite[Lemma 5.1]{Gompf} (also see \cite[Proposition 6.1]{BaykurHayano}), such a rational blowdown has the same affect as regular blowdown whenever this $(-4)$--sphere intersects an exceptional sphere once, which is the case here. (Either one of the exceptional spheres corresponding to the two boundary twists in the braiding lantern configuration hits the $(-4)$--sphere once.)  So the resulting symplectic $4$--manifold is again $X_g(0)= X'_g(0) \# \CPb$. Since $X'_g(0)$ is minimal, any exceptional sphere in $X_g'(0)\sharp \CPb$ is homologous (up to sign) to this bisection, so it should intersect a regular fiber algebraically and geometrically twice. 
Thus, the Lefschetz fibration we obtained after the braiding lantern substitution cannot admit any exceptional sections.

Next, assume $g$ is odd. Once again, the manifold $X_g'(1)$ is spin by Lemma~\ref{T:spin X_g'(i)}, and therefore minimal.  As in the previous paragraph, we can show that $\kappa(X_g(1))=1$  if $g\geq 4$. Whereas, by Theorem~\ref{T:Kodaira dim X_g(i)},$\kappa(X_3(1))=0$. The manifold $X_g(1)$ admits a Lefschetz fibration with the following monodromy factorization: 
{\allowdisplaybreaks
\begin{align*}
t_{\delta_{2}}t_{\delta_1} t_{\delta_{2}^\prime} t_{\delta_1^\prime}= &t_{x_{2}}t_{x_1} t_{x_{2}^\prime}t_{x_1^\prime} (t_{c_1} t_{c_2} t_{c_3})^{4g-4} D_g E_g \\
\sim & \underline{t_{x_{2}}t_{x_{2}^\prime}t_{c_1}^2}\cdot \underline{t_{x_1} t_{x_1^\prime}t_{c_3}^2} t_{t_{c_1}^{-1}t_{c_3}^{-2}(c_2)} t_{t_{c_3}^{-1}(c_2)} (t_{c_1} t_{c_2} t_{c_3})^{4g-6} D_g E_g 
\end{align*}
}
As we did above, we can apply braiding lantern substitution at the underlined parts. 
The resulting Lefschetz fibration has two exceptional bisections, and the total space of it is $X_g'(1)\sharp 2\CPb$, by the same arguments as above.
Since $X_g'(1)$ is minimal and has non-negative Kodaira dimension, the homology classes represented by the two bisections are the only classes (up to sign) represented by exceptional spheres. 
Thus, the Lefschetz fibration on $X_g'(1)\sharp 2\CPb$ cannot admit any exceptional sections.

Finally, by Usher's theorem on minimality of symplectic fiber sums \cite{UsherMinimality} (also see \cite{BaykurMinimality}) the presence of exceptional spheres in total spaces of all the Lefschetz fibrations above imply that none can be a fiber sum of nontrivial Lefschetz fibrations. 
\end{proof}

\begin{remark}
There are many more counter-examples one can produce using similar arguments and ingredients. Note that the product $t_{x_{i+1}}\cdots t_{x_{1}}t_{x_{i+1}'}\cdots t_{x_1'}$ is Hurwitz equivalent to $t_{x_{i+1}}t_{x_{i+1}'}\cdots t_{x_1}t_{x_1'}$. We can change the product \[(t_{c_1}t_{c_2}t_{c_3})^{4(g-i)-l}(t_{c_3}t_{c_2}t_{c_1})^l \  \ (\text{where } l=0  \text{ or } 2)\]
to $t_{c_1}^{4(g-i)}t_{c_3}^{4(g-i)}W$ via Hurwitz moves, where $W$ is some product of Dehn twists.
Thus, as we did in the proof of Theorem~\ref{T:counterex Stipsicz}, we can apply braiding lantern substitution to the Lefschetz fibration on $X_g(i)$ obtained in Section~\ref{Sec:newpencils} so that the resulting Lefschetz fibration has $i+1$ exceptional bisections, when $i+1$ is less than or equal to $4(g-i)$.  If $g+i$ is even and $i\leq g-2$, then we can further prove that the resulting fibration cannot admit any exceptional sections. Hence, we record that at least for any $g\leq 9$, there exists a genus--$g$ Lefschetz fibration $(X,f)$, which is a counter-example to the Stipsicz conjecture, where $\kappa(X)=0$.
\end{remark}

\medskip

\section{Further applications } \label{Sec:further}

In this last section, we will show that by combining unchaining surgery with rational blowdowns, 
we can produce further interesting examples. The two applications we present here, one regarding the topology of symplectic $4$--manifolds and one regarding that of pencils, will utilize the Lefschetz pencils $X'_g(i)$ we constructed using unchaining. 

\subsection{Exotic $4$--manifolds with $b^+=3$ via genus-$3$ fibrations}\label{Sec:further1} \

In this subsection we will prove the following theorem:

\begin{theorem}\label{exotic}
There are  genus--$3$ Lefschetz fibrations $(X_j, f_j)$, for   $j=0,1,2, 3$, where each $X_j$ is a minimal symplectic $4$--manifold homeomorphic but not diffeomorphic to $3\, \CP \sharp \,(19-j)\, \CPb$, and  each $f_{j+1}$ has a monodromy factorization obtained from that of $f_{j}$ by a lantern substitution.
\end{theorem}

We will need the following lemma in our proof:
\begin{lemma}\label{sevenLS}
Let the curves $c_j, a, a'$ be as in Figures\ref{F:curves1} and\ref{vanishingcycles}. One can perform $7$ consecutive lantern substitutions within the product  $t_a^2 t_{a^\prime}^2 t_{c_1}^6 t_{c_3}^3 t_{c_5}^3 t_{c_7}^6$ in $\Gamma_3$, up to Hurwitz moves. 
\end{lemma}

\begin{figure}[hbt]
 \centering
     \includegraphics[width=8cm]{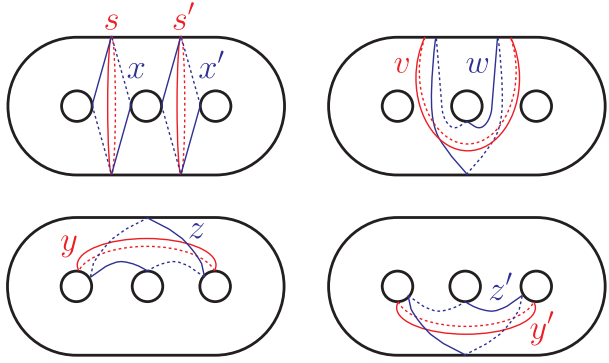}
     \caption{The curves in a genus--$3$ surface.}
     \label{genus3}
\end{figure}
\begin{proof}
{
Let $x,x^\prime,y,y^\prime,z,z^\prime,w,s,s^\prime,v$ be the curves given in $\Sigma_3$ as in Figure\ref{genus3}. 
Note that $x,x^\prime,y,y^\prime,z,z^\prime,w$ are non-separating and $s,s^\prime,v$ are separating. 
We have the following lantern relators $L_1,L_1^\prime,L_2,L_2^\prime$ and $L_3$:
\begin{align*}
&L_1 = t_{c_3}t_st_x t_{a}^{-1}t_{c_1}^{-1}t_{c_1}^{-1}t_{a^\prime}^{-1},& 
L_1^\prime = t_{c_5}t_{s^\prime}t_{x^\prime} t_{c_7}^{-1}t_{a}^{-1}t_{a^\prime}^{-1}t_{c_7}^{-1}, \\
&L_2 = t_at_yt_z t_{c_1}^{-1}t_{c_5}^{-1}t_{c_3}^{-1}t_{c_7}^{-1},
&L_2^\prime = t_{a^\prime}t_{y^\prime}t_{z^\prime} t_{c_1}^{-1}t_{c_5}^{-1}t_{c_3}^{-1}t_{c_7}^{-1}, \\
&L_3 =t_{a^\prime}t_vt_w t_a^{-1}t_s^{-1}t_{s^\prime}^{-1}t_a^{-1}.
\end{align*}
}

For $W_1$ and $W_2$ any two products of Dehn twists in $\Gamma_g$, we will write $W_1 \xrightarrow[]{L} W_2$ when $W_2$ is obtained by applying a lantern substitution to $W_1$ using the lantern relator  $L$. 
Since 
\[t_a^2 t_{a^\prime}^2 t_{c_1}^6 t_{c_3}^3 t_{c_5}^3 t_{c_7}^6 \sim t_{c_1}^4 t_{c_3}^3 t_{c_5}^3 t_{c_7}^4 \cdot t_{c_7}t_{a^\prime}t_{a}t_{c_7} \cdot t_{a^\prime}t_{c_1}t_{c_1}t_{a},\]
{
applying $L_1$-substitutions gives 
\begin{align*}
t_{c_1}^4 t_{c_3}^3 t_{c_5}^3 t_{c_7}^4  \cdot t_{c_7}t_{a^\prime}t_{a}t_{c_7} \cdot t_{a^\prime}t_{c_1}t_{c_1}t_{a} &\xrightarrow[]{L_1} t_{c_1}^4 t_{c_3}^3 t_{c_5}^3 t_{c_7}^4 \cdot t_{c_7}t_{a^\prime}t_{a}t_{c_7} \cdot t_{c_3}t_st_x. \end{align*}
Moreover, because
\[
t_{c_1}^4 t_{c_3}^3t_{c_5}^3 t_{c_7}^4 \cdot t_{c_7}t_{a^\prime}t_{a}t_{c_7} \cdot t_{c_3}t_st_x \sim (t_{c_7} t_{c_5} t_{c_3} t_{c_1})^3\cdot t_{c_1}t_{c_3}t_{c_7} \cdot t_{c_7}t_{a^\prime}t_{a}t_{c_7} \cdot t_st_x,
\]
by applying $L_2$-substitutions twice and $L_2^\prime$-substitution once we obtain
\begin{align*}
& (t_{c_7} t_{c_5} t_{c_3} t_{c_1})^3\cdot t_{c_1}t_{c_3}t_{c_7} \cdot t_{c_7}t_{a^\prime}t_{a}t_{c_7} \cdot t_st_x \\ \xrightarrow[]{L_2,L_2,L_2^\prime}& (t_at_yt_z)^2(t_{a^\prime}t_{y^\prime}t_{z^\prime}) \cdot t_{c_1}t_{c_3}t_{c_7} \cdot t_{c_7}t_{a^\prime}t_{a}t_{c_7} \cdot t_st_x. 
\end{align*}
We can further apply $L_1^\prime$ and $L_2^\prime$-substitutions as follows:
\begin{align*}
&(t_at_yt_z)^2(t_{a^\prime}t_{y^\prime}t_{z^\prime}) \cdot t_{c_1}t_{c_3}t_{c_7} \cdot t_{c_7}t_{a^\prime}t_{a}t_{c_7} \cdot t_st_x \\
\xrightarrow[]{L_1^\prime} &
(t_at_yt_z)^2(t_{a^\prime}t_{y^\prime}t_{z^\prime}) \cdot t_{c_1}t_{c_3}t_{c_7} \cdot t_{c_5}t_{s^\prime}t_{x^\prime} \cdot t_st_x \\
\xrightarrow[]{L_2^\prime} &
(t_at_yt_z)^2(t_{a^\prime}t_{y^\prime}t_{z^\prime})^2 \cdot t_{s^\prime}t_{x^\prime} \cdot t_st_x. 
\end{align*}
Here, note that $t_a t_y t_z \sim t_y t_z t_a$, which follows from the lantern relation $t_at_yt_z=t_{c_1}t_{c_3}t_{c_5}t_{c_7}$, as $a$ is disjoint from $c_1,c_3,c_5,c_7$. 
Therefore, we may further rewrite the product above as
\begin{align*}
&(t_at_yt_z)^2(t_{a^\prime}t_{y^\prime}t_{z^\prime})^2 \cdot t_{s^\prime}t_{x^\prime} \cdot t_st_x \\
 \sim& (t_yt_z)^2(t_{a^\prime}t_{y^\prime}t_{z^\prime})^2 \cdot t_a^2 t_st_{s^\prime} \cdot t_xt_{x^\prime}. 
\end{align*}
Finally, by applying $L_3$ we obtain 
\begin{align*}
(t_yt_z)^2(t_{a^\prime}t_{y^\prime}t_{z^\prime})^2 \cdot t_a^2 t_st_{s^\prime} \cdot t_xt_{x^\prime} \xrightarrow[]{L_3} (t_yt_z)^2(t_{a^\prime}t_{y^\prime}t_{z^\prime})^2 \cdot t_{a^\prime} t_vt_w \cdot t_xt_{x^\prime}.
\end{align*}
This finishes the proof. }
\end{proof}

The proof of the theorem will now follow from applying the above lemma to the positive factorization for the Lefschetz fibration $X_3(1)$ we obtained by unchaining.

\begin{proof}[Proof of Theorem~\ref{exotic}]
Let $(X,f)$ be the genus--$3$ Lefschetz fibration $(X_3(1), f_3(1))$ on $4$ times blow-up of a symplectic Calabi-Yau homotopy $\K$--surface, whose monodromy is: 
\begin{align*}
t_{\delta_2} t_{\delta_1} t_{\delta_2'} t_{\delta_1'} =  t_{x_2} t_{x_1} t_{x_2^\prime} t_{x_1^\prime} (t_{c_1} t_{c_2} t_{c_3})^8 t_{d_4} t_{d_5} t_{d_6} t_{d_7} t_{e_7} t_{e_6} t_{e_5} t_{e_4}. 
\end{align*}
We will show that this monodromy factorization is Hurwitz equivalent to a positive factorization which contains the product in Lemma~\ref{sevenLS}. 

We can check at once that $t_{d_4}t_{d_5}t_{d_6}t_{d_7}(c_{i+4})=c_{i}$ for $i=1,2,3$.
This gives 
\begin{align*}
t_{c_1}t_{c_2}t_{c_3} \cdot t_{d_4}t_{d_5}t_{d_6}t_{d_7} \sim  t_{d_4}t_{d_5}t_{d_6}t_{d_7} \cdot t_{c_5}t_{c_6}t_{c_7}. 
\end{align*}
Therefore, in $\Gamma_3$, the following relation holds
\begin{align}
1 &= t_a^2 t_{a^\prime}^2 (t_{c_1} t_{c_2} t_{c_3})^8 t_{d_4} t_{d_5} t_{d_6} t_{d_7} t_{e_7} t_{e_6} t_{e_5} t_{e_4} \label{genus3case} \\
&\sim  t_a^2 t_{a^\prime}^2 (t_{c_1} t_{c_2} t_{c_3})^4  t_{d_4} t_{d_5} t_{d_6} t_{d_7} (t_{c_5} t_{c_6} t_{c_7})^4 t_{e_7} t_{e_6} t_{e_5} t_{e_4}.  \notag
\end{align}
Here, using the relation (\ref{hurwitz2}) we obtain 
\begin{align*}
t_{c_1}t_{c_2}t_{c_3} \cdot t_{c_1}t_{c_2}t_{c_3} \cdot t_{c_1}t_{c_2}t_{c_3} \cdot t_{c_1}t_{c_2}t_{c_3} &\sim t_{c_1} \cdot t_{c_1}t_{c_2}t_{c_3} \cdot t_{c_1}t_{c_2} \cdot t_{c_1}t_{c_2}t_{c_3} \cdot t_{c_1}t_{c_2}t_{c_3} \\
&\sim t_{c_1} \cdot t_{c_1}t_{c_2}t_{c_3} \cdot t_{c_1} \cdot t_{c_1}t_{c_2}t_{c_3} \cdot t_{c_1} \cdot t_{c_1}t_{c_2}t_{c_3}. 
\end{align*}
That is, we have
\begin{align*}
(t_{c_1}t_{c_2}t_{c_3})^4 \sim (t_{c_1}^2t_{c_2}t_{c_3})^3 \sim t_{c_1}^6 t_{c_3}^3 t_{t_{c_1}^{-4}t_{c_3}^{-3}(c_2)} t_{t_{c_1}^{-2}t_{c_3}^{-2}(c_2)} t_{t_{c_3}^{-1}(c_2)}. 
\end{align*}
Similarly, we get
\begin{align*}
(t_{c_5}t_{c_6}t_{c_7})^4 \sim (t_{c_5}t_{c_6}t_{c_7}^2)^3 \sim t_{t_{c_5}(c_6)} t_{t_{c_5}^2t_{c_7}^2(c_6)} t_{t_{c_5}^3t_{c_7}^4(c_6)} t_{c_5}^3 t_{c_7}^6. 
\end{align*}
The above arguments and cyclic permutations then give 
\begin{align*}
1 &= t_a^2 t_{a^\prime}^2 (t_{c_1} t_{c_2} t_{c_3})^8 t_{d_4} t_{d_5} t_{d_6} t_{d_7} t_{e_7} t_{e_6} t_{e_5} t_{e_4}  \\
&= t_a^2 t_{a^\prime}^2 t_{c_1}^6 t_{c_3}^3 t_{c_5}^3 t_{c_7}^6 \cdot t_{t_{c_1}^{-4}t_{c_3}^{-3}(c_2)} t_{t_{c_1}^{-2}t_{c_3}^{-2}(c_2)} t_{t_{c_3}^{-1}(c_2)} \cdot t_{d_4}t_{d_5}t_{d_6}t_{d_7} \cdot T_2,
\end{align*}
where $T_2$ is a product of (seven) positive Dehn twists along nonseparating curves. 
{
As we applied in the proof of Lemma~\ref{sevenLS}}, we can apply $7$ consecutive Lantern substitutions to the monodromy of $(X, f)$, which yield new symplectic genus--$3$ Lefschetz fibrations $(X_k, f_k)$, for  $k=1, \ldots 7$, where $X_k$ is obtained from $X$ by rationally blowing-down $k$ symplectic $(-4)$--spheres. 
{(Note that we apply the $7$ Lantern substitutions in the same order as we did in the proof of Lemma~\ref{sevenLS}.)}

Let us first determine the homeomorphism type of each $X_k$. From the algebraic topological invariants of $(X,f)=(X_3(1), f_3(1))$ we calculated earlier, we deduce that $\eu(X_k)=28-k$ and $\sigma(X_k)=-20+k$. Moreover, we claim that $\pi_1(X_k)=1$ for each $k=1,2,\ldots,7$. 
When $1 \leq k \leq 5$, we see that the monodromy of $f_k$ contains Dehn twists $t_{d_4},t_{d_5},t_{d_6},t_{d_7},t_{c_1},t_{c_2},t_a$. These Dehn twist curves alone give enough relations to kill the fundamental group as in the proof of Theorem~\ref{topXg}. If $k=6,7$, the monodromy of $f_k$ contains the Dehn twists $t_{a^\prime},t_x,t_{t_{c_3}^{-1}(c_2)},t_{d_4},t_{d_5},t_{d_6},t_{d_7}$.  Since $x=a_1a_2=1$, $t_{c_3}^{-1}(c_2)=b_1^{-1}a_2^{-1}b_1a_1b_1^{-1}=1$ and $a^\prime = [a_1,b_1]a_2$, we again have $\pi_1(X_j)=1$. None of these manifolds have even intersection forms, because they all contain reducible fibers, which always yield surfaces of odd self-intersection. (And none other than $X_4$ has signature divisible by $16$ anyway.) Hence, by Freedman's theorem, we see that each $X_k$ is homeomorphic to $3\CP\sharp (23-k)\CPb$.

On the other hand, the standard $4$--manifolds $3\CP\sharp (23-k)\CPb$ do not admit symplectic structures: any $4$--manifold that is a smooth connected sum of two \mbox{$4$--manifolds} with $b_2^+>0$ has vanishing Seiberg-Witten invariants, whereas by Taubes, any symplectic $4$--manifold with $b^+>0$ has non-trivial Seiberg-Witten invariants. So, none of the $X_k$ is diffeomorphic to $3\, \CP\sharp \,  (23-k)\, \CPb$. 

The more interesting cases will be for $k=4, 5, 6,7$. (When $k \leq 3$, $c_1^2(X_k)=-4+k <0$, whereas by Taubes' seminal work, a minimal symplectic $4$--manifold with $b_2^+>0$ always has $c_1^2 \geq 0$.) The rest of our proof is devoted to showing that each $X_k$, for $k=4,5,6,7$, is indeed minimal. To do so, we will need to go over our construction one more time, this  time paying attention to how exceptional spheres intersect the spheres cobounded by the $4$ lantern curves in each lantern substitution. (After a small perturbation, each one of these spheres can be contained in a singular fiber with multiple nodes, and it is this symplectic $(-4)$--sphere one rationally blowdowns in the course of the lantern substitution.) Below we will simply refer to these as \emph{lantern spheres}, and denote them using the corresponding lantern substitutions (while pointing out any potential ambiguities).

The initial genus--$3$ Lefschetz fibration $(X,f)$ we started with, where $X$ is the $4$ times blow-up of a symplectic Calabi-Yau homotopy $\K$ surface we will simply denote by $X'$, had a total of $4$ exceptional sections $S_1, S_2, S'_1, S'_2$, which hit the fibers at the marked points obtained by collapsing the boundary components $\delta_1, \delta_2, \delta'_1, \delta'_2$ in Figure~\ref{sectioncurves}. The reader may want to refer to this figure and Figure~\ref{F:curves1} for the rest of our discussion, where the curves $x_1, x_2$ and $x'_1, x'_2$ are the lifts of $a$ and $a'$ in our lantern substitutions. 

{
The first lantern sphere $L_1$ we rationally blowdown can be seen to intersect once with $S_2, S'_2$. 
So by \cite[Lemma 5.1]{Gompf}, $X_1$ is still diffeomorhic to $3$ times blow-up of $X'$. 
Moreover, by \cite[Lemma 6.1]{BaykurHayano}, the substitution, when the marked points are taken into account, correspond to a braiding lantern substitution, which turns the pair of exceptional sections the lantern sphere intersects to an exceptional bisection of the resulting fibration. 
Denote this exceptional bisection by $S_{22'}$. 
Let us denote the next three lantern spheres, two of which correspond to the lantern substitution $L_2$ while one of which corresponds to $L_2^\prime$, as $L_2(1)$, $L_2^\prime(2)$ and $L_2(3)$ in the order they will be blown down. 


The next lantern sphere $L_2(1)$ intersects the exceptional spheres $S_{1}$ and $S_{22'}$ only once, and therefore we once again conclude that $X_2$ is diffeomorphic to a twice blow-up of $X'$, whereas the spheres $S_1$ and $S_{22'}$ now descend to an exceptional triple-section $S_{122'}$. 
Similarly, we can verify that $X_3$ is diffeomorphic to a single blow-up of $X'$ and the exceptional sphere, denoted by $S_{11'22'}$, becomes a quadruple-section. 
Now, the lantern sphere $L_2(3)$ intersects $S_{11'22'}$ algebraically and geometrically twice.}
So by \cite[Lemma~1.1]{Dorfmeister1}, the symplectic $4$--manifold $X_4$ obtained by rationally blowing down $L_2(3)$ in $X_3$ is minimal, and moreover, the symplectic $4$--manifolds $X_5, X_6, X_7$ we obtain by further rational blowdowns along the $(-4)$--spheres remain to be minimal. Lastly, for the satement of the theorem, we only record the minimal cases (when $k=4,5,6,7$) by resetting the index of $X_j$ as $j=k-4$.
\end{proof}

\smallskip
\begin{remark} \label{whereiscan}
With all the additional information we have on the exceptional (multi)sections of the Lefschetz fibrations $(X_k, f_k)$, we can easily calculate the Seiberg-Witten invariants of $X_k$, and moreover, represent their basic classes as multisections. The symplectic Calabi-Yau surface $X'$ has trivial canonical class, which is its only Seiberg-Witten basic class \cite{MorganSzabo}. Therefore the Seiberg-Witten invariants of each $X_k \cong X' \#  (4-k) \CPb$, for $k=0, 1,2,3$, are determined by the blow-up formula \cite{Witten}, and the $2^{4-k}$ basic classes of $X_0:=X$, $X_1$, $X_2$ and $X_3$ are all represented by the collection of exceptional (multi)sections $\{S_1, S'_1, S_2, S'_2\}$, {$\{S_{1}, S_{1'}, S'_{22'}\}$, $\{S_{1}, S_{122'}\}$, and $\{S_{1122'}\}$}, respectively, where each sphere is taken with either orientation. In particular, 
$X_3 \cong X' \# \CPb$ has only two basic classes: $\pm E$, where $E$ is  represented by $S_{11'22'}$. In turn, by the Seiberg-Witten formula for rational blowdown 
\cite{FSrationalblowdown}, the only basic class that descends from $X_3$ to $X_4$ is $E$, since it intersects the $(-4)$--sphere of the rational blowdown exactly $4$ times, and the class it descends to, up to sign, is the only basic class of $X_4$.  Using the machinery of \cite{BaykurHayano}, we can conclude that this new class is represented by a quadruple-section (represented by the same $4$ marked points on the fiber), which is now a torus $R_0$ of self-intersection $0$. Iterating the same arguments, we see that each $X_{4+j}$, for $j=0, 1, 2, 3$,  has only one basic class up to sign, which is represented by a quadruple-section $R_j$ of genus $j$, with $R_j^2= j$. By Taubes, the multisection $R_j$ represents the symplectic canonical class (and its Seiberg-Witten value is $+1$).  
\end{remark}

\medskip
\subsection{Inequivalent pencils with explicit monodromies}\label{Sec:further2} \

Here we will construct inequivalent relatively minimal genus--$g$ Lefschetz fibrations for each $g \geq 3$, on a family of symplectic $4$--manifolds of Kodaira dimension $-\infty, 0$ and $1$, and especially on rational surfaces. 

\begin{theorem}\label{thm:5}
For any $g\geq 3$ and $i=0,1,2,\ldots,g-1$, there are pairs of inequivalent relatively minimal  genus--$g$ Lefschetz pencils $(Y_g(i), h^j_g(i))$, $j=1,2$, and inequivalent Lefschetz fibrations on their blow-ups. For any $g\geq 3$, there are pairs of  inequivalent  relatively minimal genus--$g$ Lefschetz pencils on once blown-up elliptic surface $E(1) \, \sharp \, \CPb \cong \CP\, \sharp \, 10\, \CPb$, with different number of reducible fibers.
\end{theorem}

These inequivalent pencils and fibrations will be produced following the recipe given in \cite{BaykurHayano}, and will be distinguished by the number of their reducible fibers. Therefore they are even inequivalent up to  \emph{fibered Luttinger surgeries} (see 
\cite{AurouxLuttinger, BaykurInequivalentLF, BaykurInequivalentRR}), which never change the topological type of fibers.


We will need the following lemma:

\begin{lemma}\label{lemA}
Let $x_1^{\prime\prime}$ be the simple closed curve on $\Sigma_g^{2(i+1)}$ as in Figure\ref{curvezeta1}, and let $D_{g}^\prime:=t_{d_6}t_{d_7}\cdots t_{d_{2g+1}}$ and $E_g^\prime := t_{e_{2g+1}} \cdots t_{e_7}t_{e_6}$. 
For each $i=0,1,2,\ldots,g-1$, the following positive factorization is Hurwitz equivalent to the monodromy factorization of the pencil $(X'_g(i), f'_g(i))$ given in Theorem~\ref{thm:3}: 

\noindent when $g$ is odd, 
\begin{align*}
&t_{\delta_{i+1}} \cdots t_{\delta_2} t_{\delta_1} t_{\delta_{i+1}^\prime} \cdots t_{\delta_2^\prime} t_{\delta_1^\prime} \\
&= \prod_{k=i+2}^{2}  t_{x_k} t_{x_k^\prime} \cdot t_{c_2} t_{c_3} (t_{c_1} t_{c_2} t_{c_3})^{4(g-1-i)+1} (t_{c_1} t_{c_2} t_{c_3})^2 \cdot t_{x_1} t_{d_4} t_{d_5} t_{x_1^{\prime\prime}} \cdot D_g^\prime E_g^\prime \cdot t_{c_5} \cdot  t_{e_5} t_{e_4}
\end{align*}
when $g$ is even, 
\begin{align*}
&t_{\delta_{i+1}} \cdots t_{\delta_2} t_{\delta_1} t_{\delta_{i+1}^\prime} \cdots t_{\delta_2^\prime} t_{\delta_1^\prime} \\
&= \prod_{k=i+2}^{2}  t_{x_k} t_{x_k^\prime} \cdot t_{c_2} t_{c_3} (t_{c_1} t_{c_2} t_{c_3})^{4(g-1-i)+1} (t_{c_3} t_{c_2} t_{c_1})^2 \cdot  t_{x_1} t_{d_4} t_{d_5} t_{x_1^{\prime\prime}} \cdot D_g^\prime E_g^\prime \cdot t_{c_5} \cdot  t_{e_5} t_{e_4}. 
\end{align*}
\end{lemma}
\begin{figure}[hbt]
 \centering
      \includegraphics[scale=.50]{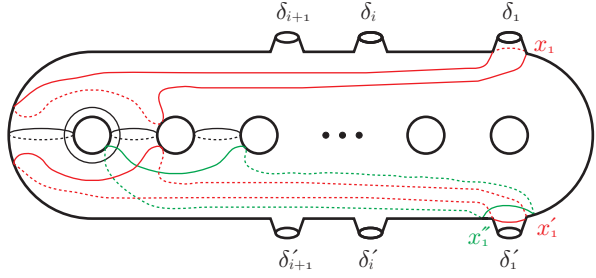}
      \caption{The curve $x_1^{\prime\prime}$ on $\Sigma_g^{2(i+1)}$.}
      \label{curvezeta1}
 \end{figure}

\begin{proof}
Let us prove the case of even $g$, and leave the odd case to the reader, the proof of which is very similar.

Note that since $c_1,c_2,c_3$ are disjoint from $x_1,\ldots,x_i$ and $x^\prime_1,\ldots,x^\prime_i$, and $x_\ell$ is disjoint from $x^\prime_j$ for any $\ell$ and $j$, we have the following relation:
\begin{align*}
&t_{x_{i+1}} \cdots t_{x_2} t_{x_1} t_{x_{i+1}^\prime} \cdots t_{x_2^\prime} t_{x_1^\prime} (t_{c_1} t_{c_2} t_{c_3})^{4(g-1-i)+2} (t_{c_3} t_{c_2} t_{c_1})^2 D_g E_g \\
&\sim t_{c_1} \cdot \prod_{k=i+1}^2 t_{x_k} t_{x_k^\prime} \cdot t_{c_2}t_{c_3} (t_{c_1} t_{c_2} t_{c_3})^{4(g-1-i)+1} (t_{c_3} t_{c_2} t_{c_1})^2 \cdot t_{x_1} t_{x_1^\prime} \cdot D_g E_g \\
&\sim \prod_{k=i+1}^2 t_{x_k} t_{x_k^\prime} \cdot t_{c_2}t_{c_3} (t_{c_1} t_{c_2} t_{c_3})^{4(g-1-i)+1} (t_{c_3} t_{c_2} t_{c_1})^2 \cdot t_{x_1} t_{x_1^\prime} \cdot D_g E_g \cdot t_{c_1}
\end{align*}
where $D_g=t_{d_4} t_{d_5} \cdots t_{d_{2g+1}}$ and $E_g=t_{e_{2g+1}} \cdots t_{e_5} t_{e_4}$. 
Here, it is easy to check that $x_1^{\prime\prime}=(t_{d_4}t_{d_5})^{-1}(x_1^\prime)$ and $c_5=(t_{e_5}t_{e_4})(c_1)$. 
This gives 
\begin{align*}
&t_{x_1^\prime} \cdot t_{d_4}t_{d_5} \sim t_{d_4}t_{d_5} \cdot t_{x_1^{\prime\prime}} \\
&t_{e_5}t_{e_4} \cdot t_{c_1} \sim t_{c_5} \cdot t_{e_5}t_{e_4}. 
\end{align*}
This completes the proof.
\end{proof}

\begin{figure}[hbt]
 \centering
      \includegraphics[scale=.47]{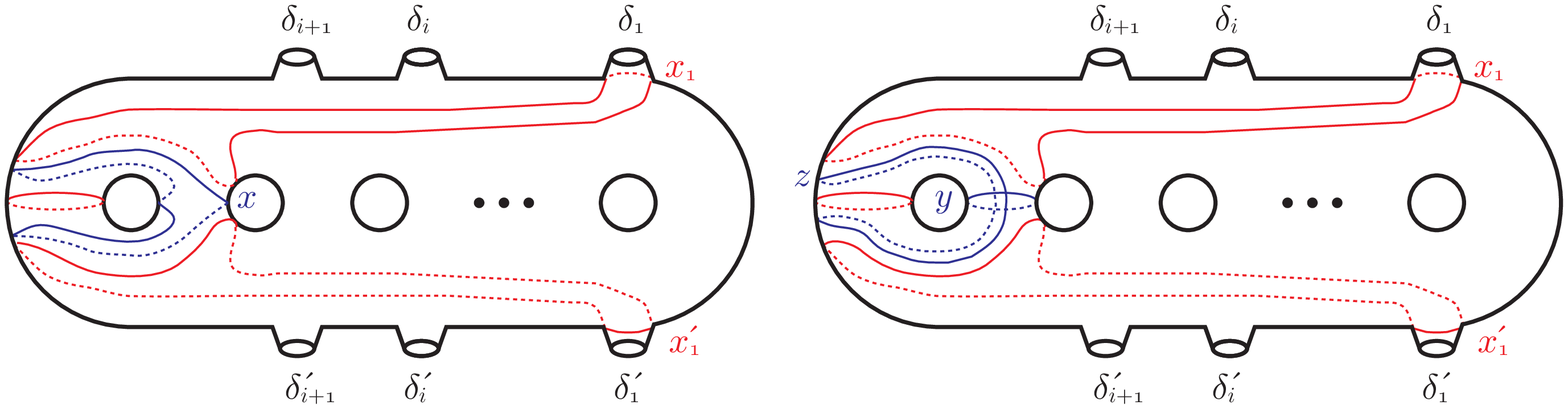}
      \caption{The curves $x,y,z$ on $\Sigma_g^{2(i+1)}$.}
      \label{curves4}
 \end{figure}

\medskip
\begin{proof}[Proof of Theorem~\ref{thm:5}] We will run our constructions for any $i=0,1,2,\ldots,g-1$. 

It is easy to see that, after Hurwitz moves, the monodromy factorization of $(X'_g(i), f_g(i))$ yields a positive factorization with the subword $t_{c_1}^2 t_{x_1}t_{x_1}^\prime$. 
Note that $x_1,x_1^\prime,\delta_1,\delta_1^\prime$ and two disjoint curves parallel to $c_1$ bound a sphere with six boundary components.  Applying the lantern substitution to this subword $t_{c_1}^2t_{x_1}t_{x_1^\prime}$ replaces it with the new subword $t_{x}t_{y}t_{z}$ in the factorization, where $x,y,z$ are the curves as in Figure~\ref{curves4}. Akin to  the proof of Theorem~\ref{T:counterex Stipsicz}, this then gives a genus--$g$ Lefschetz pencil $(Y^1_g(i), h^1_g(i))$. This pencil $h^1_g(i)$ has one reducible fiber corresponding to the Dehn twist along the separating curve $y$.

\begin{figure}[hbt]
 \centering
      \includegraphics[scale=.47]{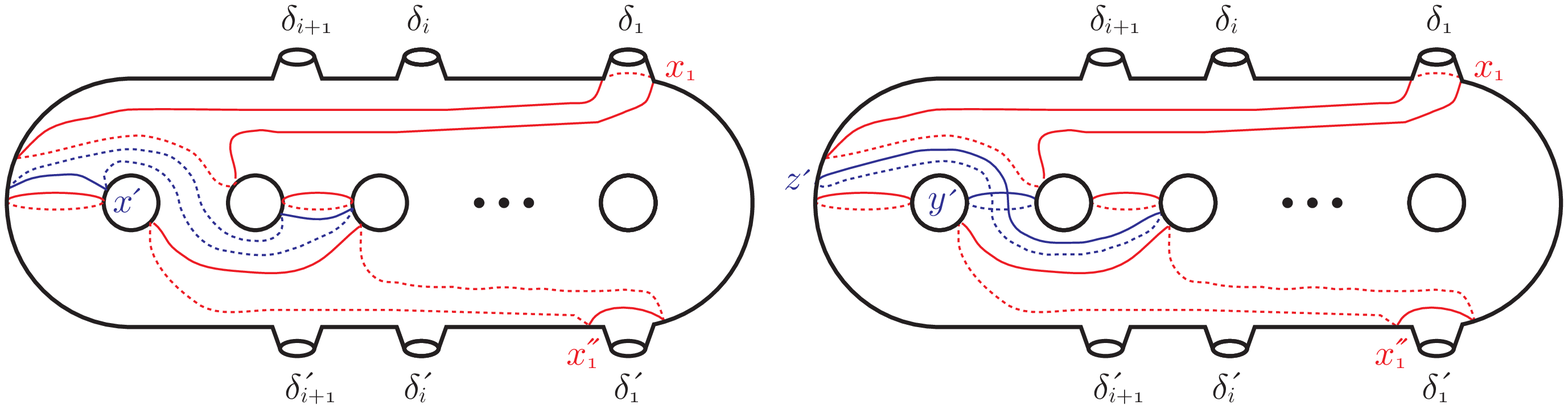}
      \caption{The curves $x^\prime,y^\prime,z^\prime$ on $\Sigma_g^{2(i+1)}$.}
      \label{curves5}
 \end{figure}

On the other hand, if we take the Hurwitz equivalent monodromy factorization of $(X'_g(i), f_g(i))$ given in Lemma~\ref{lemA}, it is again easy to see that, after further Hurwitz moves, we get a positive factorization with the subword $t_{c_1}t_{x_1}t_{x_1^{\prime\prime}}t_{c_5}$. 
Note that $c_1,x_1,x_1^{\prime\prime},c_5,\delta_1,\delta_1^\prime$ bound a sphere with six boundary components.  When we apply a braiding lantern substitution along this subword $t_{c_1} t_{x_1} t_{x_1^{\prime\prime}} t_{c_5}$, we replace it with the new subword $t_{x^\prime}t_{y^\prime}t_{z^\prime}$ in the factorization, where $x^\prime,y^\prime,z^\prime$ are the nonseparating curves in Figure~\ref{curves5}. This gives a genus--$g$ Lefschetz pencil $(Y^2_g(i), h^2_g(i))$ where  as well, where all singular fibers of $h^2_g(i)$ are irreducible.

By the arguments we employed in earlier proofs, since the lantern spheres of for these substitutions intersect an exceptional section once, we have $Y^j_g(i) \cong X_g^\prime(i) \sharp\CPb$, for each $j=1, 2$. Therefore, the two pencils $h^j_g(i)$ can be viewed on the same symplectic $4$--manifold, but have different number of reducible fibers (one versus none), which means they are not equivalent.

The examples in the last part of the theorem are derived in the particular case when $i=g-1$,  where $Y_g(g-1) \cong X_g^\prime(g-1) \sharp\CPb$, and $X'_g(g-1) \cong E(1)$ by Theorem~\ref{topXg}. 
\end{proof}

\end{document}